\documentclass[journal]{IEEEtran}
\ifCLASSINFOpdf
\usepackage[pdftex]{graphicx}
\else
\fi
\usepackage{amsmath, amssymb, amsthm}
\usepackage{amsfonts}
\usepackage{algorithmic}
\ifCLASSOPTIONcompsoc
\usepackage[caption=false,font=normalsize,labelfont=sf,textfont=sf]{subfig}
\else
\usepackage[caption=false,font=footnotesize]{subfig}
\fi

\usepackage{color}
\usepackage{xcolor}
\usepackage{tabularx}
\usepackage{multirow}

\newcommand{\mW}{\mathbf{W}}
\newcommand{\mA}{\mathbf{A}}
\newcommand{\mx}{\mathbf{x}}

\newcommand{\ox}{\overline{x}}
\newcommand{\on}{\overline{n}}

\newcommand{\og}{\overline{g}}
\newcommand{\bE}{\mathbb{E}}
\newcommand{\mg}{\mathbf{g}}

\newcommand{\hV}{\hat{V}}

\newcommand{\T}{\intercal}

\def\sp#1{\color{black}{#1}}

\definecolor{darkgreen}{rgb}{0.01, 0.75,0.24}

\newtheorem{assumption}{Assumption}
\newtheorem{lemma}{Lemma}
\newtheorem{theorem}{Theorem}
\newtheorem{corollary}{Corollary}

\usepackage{latexsym}
\usepackage{amsfonts}
\theoremstyle{remark}
\newtheorem{remark}{Remark}

\newcommand{\mF}{\mathcal{F}}

\begin{document}
	%
	% paper title
	% Titles are generally capitalized except for words such as a, an, and, as,
	% at, but, by, for, in, nor, of, on, or, the, to and up, which are usually
	% not capitalized unless they are the first or last word of the title.
	% Linebreaks \\ can be used within to get better formatting as desired.
	% Do not put math or special symbols in the title.
	\title{A Sharp Estimate on the Transient Time of Distributed Stochastic Gradient Descent}
	%
	%
	% author names and IEEE memberships
	% note positions of commas and nonbreaking spaces ( ~ ) LaTeX will not break
	% a structure at a ~ so this keeps an author's name from being broken across
	% two lines.
	% use \thanks{} to gain access to the first footnote area
	% a separate \thanks must be used for each paragraph as LaTeX2e's \thanks
	% was not built to handle multiple paragraphs
	%
	
	\author{Shi~Pu,
		Alex~Olshevsky,
		and~Ioannis~Ch.~Paschalidis% <-this % stops a space
		\thanks{This work was partially supported by the NSF under grants IIS-1914792, DMS-1664644, CNS-1645681, and ECCS-1933027, by the ONR under MURI grant N00014-19-1-2571, by the NIH under grants R01 GM135930 and UL54 TR004130, by the DOE under grant DE-AR-0001282, by the Shenzhen Research Institute of Big Data (SRIBD) under grant J00120190011, and by the NSFC under grant 62003287.}
		\thanks{S. Pu is with the School of Data Science, Shenzhen Research
			Institute of Big Data, The Chinese University of Hong Kong, Shenzhen, China.  (e-mail: pushi@cuhk.edu.cn).}% <-this % stops a space
		\thanks{A. Olshevsky and I. Ch. Paschalidis are with the Department of Electrical and Computer Engineering and the Division of Systems Engineering, Boston
			University, Boston, MA (e-mails: alexols@bu.edu, yannisp@bu.edu).}% <-this % stops a space
		%\thanks{Manuscript received April 19, 2005; revised August 26, 2015.}
	}

	\maketitle
	
	% As a general rule, do not put math, special symbols or citations
	% in the abstract or keywords.
	\begin{abstract}
		This paper is concerned with minimizing the average of $n$ cost functions over a network in which agents may communicate and exchange information with each other. We consider the setting where only noisy gradient information is available. To solve the problem, we study the distributed stochastic gradient descent (DSGD) method and perform a non-asymptotic convergence analysis. For strongly convex and smooth objective functions, DSGD asymptotically achieves the optimal network independent convergence rate compared to centralized stochastic gradient descent (SGD). 
		%our main contribution is to explicitly identify the non-asymptotic convergence rate as a function of characteristics of the objective functions and the network.
		%			 (e.g., spectral gap of the mixing matrix). 
		Our main contribution is to characterize the transient time needed for DSGD to approach the asymptotic convergence rate.
		%	, which we show behaves as $K_T=\mathcal{O}\left(\frac{n}{(1-\rho_w)^2}\right)$, where $1-\rho_w$ denotes the spectral gap of the mixing matrix.
		Moreover, we construct a ``hard" optimization problem 
		%	for which we show the transient time needed for DSGD to approach the asymptotic convergence rate is lower bounded by $\Omega \left(\frac{n}{(1-\rho_w)^2} \right)$, implying 
		that proves the sharpness of the obtained result.
		Numerical experiments demonstrate the tightness of the theoretical results.
	\end{abstract}
	
	% Note that keywords are not normally used for peerreview papers.

	% For peer review papers, you can put extra information on the cover
	% page as needed:
	% \ifCLASSOPTIONpeerreview
	% \begin{center} \bfseries EDICS Category: 3-BBND \end{center}
	% \fi
	%
	% For peerreview papers, this IEEEtran command inserts a page break and
	% creates the second title. It will be ignored for other modes.
	\IEEEpeerreviewmaketitle

	\section{Introduction}
	% The very first letter is a 2 line initial drop letter followed
	% by the rest of the first word in caps.
	% 
	% form to use if the first word consists of a single letter:
	% \IEEEPARstart{A}{demo} file is ....
	% 
	% form to use if you need the single drop letter followed by
	% normal text (unknown if ever used by the IEEE):
	% \IEEEPARstart{A}{}demo file is ....
	% 
	% Some journals put the first two words in caps:
	% \IEEEPARstart{T}{his demo} file is ....
	% 
	% Here we have the typical use of a "T" for an initial drop letter
	% and "HIS" in caps to complete the first word.
	\IEEEPARstart{W}{e} consider the distributed optimization problem where a group of agents $\mathcal{N}=\{1,2,\ldots,n\}$ collaboratively seek an $x\in\mathbb{R}^p$ that minimizes the average of $n$ cost functions:
	\begin{equation}
		\min_{x\in \mathbb{R}^{p}}f(x)\left(=\frac{1}{n}\sum_{i=1}^{n}f_i(x)\right).  \label{opt Problem_def}
	\end{equation}%
	Each local cost function $f_i:\mathbb{R}^{p}\rightarrow \mathbb{R}$ is {\sp strongly convex, with Lipschitz continuous} gradient, and is known by agent $i$ only. The agents communicate and exchange information over a network.
	Problems in the form of (\ref{opt Problem_def}) find applications in multi-agent target seeking \cite{chen2012diffusion}, distributed machine learning \cite{forrester2007multi,nedic2017fast,cohen2017projected,baingana2014proximal,ying2018supervised,alghunaim2019distributed,brisimi2018federated}, and wireless networks \cite{cohen2017distributed,mateos2012distributed,baingana2014proximal}, among other scenarios.
	% and $X\subseteq \mathbb{R}^{m}$. 
	
	In order to solve (\ref{opt Problem_def}), we assume that at each iteration $k\ge 0$, the algorithm  we study is able to obtain noisy gradient estimates $g_i(x_i(k),\xi_i(k))$, where $x_i(k)$ is 
	the input for agent $i$, that satisfy the following condition.
	\begin{assumption}
		\label{asp: gradient samples}
		{\sp For all  $k\ge 0$, 
			each random vector $\xi_i(k)\in\mathbb{R}^m$ is independent across $i\in\mathcal{N}$.
					 Denote by $\mF(k)$ the $\sigma$-algebra generated by $\{x_i(0),x_i(1),\ldots,x_i(k)\mid i\in\mathcal{N}\}$. Then,
			}
		\begin{equation}
			\label{condition: gradient samples}
			\begin{split}
				& \mathbb{E}_{\xi_i(k)}[g_i(x_i(k),\xi_i(k))\mid {\sp \mF(k)}] =  \nabla f_i(x_i(k)),\\
				& \mathbb{E}_{\xi_i(k)}[\|g_i(x_i(k),\xi_i(k))-\nabla f_i(x_i(k))\|^2\mid {\sp \mF(k)}]  \\
				& \qquad\le  {\sp \sigma^2 + M\|\nabla f_i(x_i(k))\|^2,\quad\hbox{\ for some $\sigma,M>0$}.}
			\end{split}
		\end{equation}
	\end{assumption}
	Stochastic gradients appear in many machine learning problems. For example, suppose $f_i(x):=\bE_{\sp \xi_i\sim \mathcal{D}_i}[F_i(x,\xi_i)]$ represents the expected loss function for agent $i$, where $\xi_i$ are independent data samples gathered over time, {\sp and $\mathcal{D}_i$ represents the data distribution}. Then for any $x$ and $\xi_i$ {\sp sampled from $\mathcal{D}_i$}, $g_i(x,\xi_i):=\nabla F_i(x,\xi_i)$ is an unbiased estimator of $\nabla f_i(x)$. For another example, suppose $f_i(x):=(1/|\mathcal{S}_i|) \sum_{\zeta_j\in \mathcal{S}_i}F(x,\zeta_j)$ denotes an empirical risk function, where $\mathcal{S}_i$ is the local dataset for agent $i$. In this setting, the gradient estimation of $f_i(x)$ can incur noise from various sources such as {\sp minibatch random sampling of the local dataset} and {discretization \sp for reducing communication cost \cite{reisizadeh2019exact}}.
	% and finite sample size for Monte-Carlo methods~\cite{kleijnen2008design}. 
	
	Problem (\ref{opt Problem_def}) has been studied extensively in the literature under various distributed algorithms \cite{tsitsiklis1986distributed,nedic2009distributed,nedic2010constrained,lobel2011distributed,jakovetic2014fast,kia2015distributed,shi2015extra,di2016next,qu2017harnessing,nedic2017achieving,xu2017convergence,pu2020push}, among which the distributed gradient descent (DGD) method proposed in \cite{nedic2009distributed} has drawn the greatest attention.
	%Stochastic optimization methods addressing Assumption \ref{asp: gradient samples} date back to the seminal works \cite{robbins1951stochastic} and \cite{kiefer1952stochastic}. 
	Recently, distributed implementation of stochastic gradient algorithms has received considerable interest.
	%	\cite{ram2010distributed,srivastava2011distributed,duchi2012dual,bianchi2013convergence,cavalcante2013distributed,towfic2014adaptive,morral2014success,chatzipanagiotis2015augmented,chen2015learning,chen2015learning2,chatzipanagiotis2016distributed,towfic2016excess,nedic2016stochastic,lan2017communication,lian2017can,pu2017flocking,pu2018swarming,sayin2017stochastic,sirb2018decentralized,jakovetic2018convergence,pu2020distributed,olshevsky2018robust,xin2019distributed,alghunaim2019distributed}.
	Several works 
	%	\cite{chen2012limiting,morral2014success,chen2015learning,chen2015learning2,towfic2016excess,lian2017can,pu2017flocking,morral2017success,pu2018swarming,assran2018stochastic,pu2020distributed,olshevsky2018robust,koloskova2019decentralized} 
	have shown that distributed methods may compare with their centralized counterparts under certain conditions. For example, the work in \cite{chen2012limiting,chen2015learning,chen2015learning2} first showed that, with sufficiently small constant stepsize, a distributed stochastic gradient method achieves comparable performance to a centralized method in terms of the steady-state mean-square-error.
	%		A recent paper \cite{olshevsky2018robust} discussed a distributed stochastic gradient method that asymptotically performs as well as the best bounds on centralized stochastic gradient descent (SGD) subject to possible message	losses, delays, and asynchrony.
	
	{\sp Despite the aforementioned efforts, it is unclear how long, or how many iterations it takes for a distributed stochastic gradient method to reach the convergence rate of centralized SGD. The number of required iterations, called ``transient time" of the algorithm, is a key measurement of the performance of the distributed implementation.} In this work, we perform a non-asymptotic analysis for the distributed stochastic gradient descent (DSGD) method adapted from DGD and the diffusion strategy \cite{chen2012diffusion}.\footnote{Note that in \cite{chen2012diffusion} this method was called ``Adapt-then-Combine''.} 
%	{\sp By performing a non-asymptotic analysis, we are able to precisely characterize the convergence properties for DSGD} 
	In addition to showing that the algorithm asymptotically achieves the optimal convergence rate enjoyed by a centralized scheme, we precisely identify its non-asymptotic convergence rate as a function of characteristics of the objective functions and the network (e.g., spectral gap $(1-\rho_w)$ of the mixing matrix). Furthermore, we characterize the transient time needed for DSGD to achieve the optimal rate of convergence, {\sp which behaves as $\mathcal{O}\left(\frac{n}{(1-\rho_w)^2}\right)$ assuming certain conditions on the objective functions, stepsize policy and initial solutions.}
	%	demonstrated in the following corollary.
	%	\begin{cor_restate}[Transient Time]
	%		Suppose $\sum_{i=1}^{n}\|x_i(0)-x_*\|^2=\mathcal{O}(n)$ and $\sum_{i=1}^{n}\|\nabla f_i(x_*)\|^2=\mathcal{O}(n)$. It takes $K_T=\mathcal{O}\left(\frac{n}{(1-\rho_w)^2}\right)$ time for DSGD to reach the asymptotic rate of convergence, i.e., when $k\ge K_T$, we have  $\frac{1}{n}\sum_{i=1}^n\bE[\|x_i(k)-x_*\|^2]\le \sp{\frac{\theta^2(2M\sum_{i=1}^n\|\nabla f_i(x_*)\|^2+n\sigma^2)}{(2\theta-1)n^2\mu^2 k}}\mathcal{O}(1)$.
	%	\end{cor_restate}
	%	Note that $\sp{\frac{\theta^2(2M\sum_{i=1}^n\|\nabla f_i(x_*)\|^2+n\sigma^2)}{(2\theta-1)n^2\mu^2 k}}$ is the asymptotic convergence rate for SGD (see Theorem \ref{Thm: centralized}).
	%	Here $\rho_w$ denotes the spectral norm of $\mW-\frac{1}{n}\mathbf{1}\mathbf{1}^{\T}$ with $\mW$ being the mixing matrix for all the agents, $x_i(k)$ is the solution of agent $i$ at time $k$ and $x_*$ is the optimal solution. Stepsizes are set to be $\alpha_k=\frac{\theta}{\mu (k+K)}$ for some $\theta,K>1$. 
	Finally, we construct a ``hard" optimization problem for which we show the transient time needed for DSGD to approach the asymptotic convergence rate is lower bounded by $\Omega \left(\frac{n}{(1-\rho_w)^2} \right)$, implying the obtained transient time is sharp.
	These results are new to the best of our knowledge.
	
	\subsection{Related Works}
	
	We briefly discuss the related literature on (distributed) stochastic optimization.
	First of all, our work is related to stochastic approximation (SA) methods dating back to the seminal works~\cite{robbins1951stochastic} and~\cite{kiefer1952stochastic}. 
	For a strongly convex objective function $f$ with Lipschitz continuous gradients, it has been shown that the optimal convergence rate for solving problem (\ref{opt Problem_def})  is $\mathcal{O}(\frac{1}{k})$ under a diminishing stepsize policy \cite{nemirovski2009robust}.
	
	Distributed stochastic gradient methods have received much attention in the recent years.
	%	 \cite{srivastava2011distributed,duchi2012dual,bianchi2013convergence,towfic2014adaptive,morral2014success,chatzipanagiotis2015augmented,chatzipanagiotis2016distributed,nedic2016stochastic,lan2017communication,lian2017can,sayin2017stochastic,sirb2018decentralized,jakovetic2018convergence}.
	%	The paper \cite{ram2010distributed} considered minimizing the sum of nonsmooth objective functions with a common convex constraint set. It was shown that when the means of the stochastic subgradient errors diminish, mean consensus among the agents and mean convergence to the optimum function value under SA stepsizes.
	For nonsmooth convex objective functions, the work in \cite{srivastava2011distributed} considered distributed constrained optimization and established asymptotic convergence to the optimal set using two diminishing stepsizes to account for communication noise and subgradient errors, respectively.
	The paper \cite{duchi2012dual} proposed a distributed dual averaging method which exhibits a convergence rate of $\mathcal{O}(\frac{n\log k}{(1-\lambda_2(\mW))\sqrt{k}})$ under a carefully chosen SA stepsize sequence, where $\lambda_2(\mW)$ is the second largest singular value of the mixing matrix $\mW$. 
	A projected stochastic gradient algorithm was considered in \cite{bianchi2013convergence} for solving nonconvex optimization problems by combining a local stochastic gradient update and a gossip step. This work proved that consensus is asymptotically achieved and the solutions converge to the	set of KKT points with SA stepsizes.
	In \cite{towfic2014adaptive}, the authors proposed an adaptive diffusion algorithm based on penalty methods and showed that the expected optimization error is bounded by $\mathcal{O}(\alpha)$ under a constant stepsize $\alpha$.
	The work in \cite{chatzipanagiotis2016distributed} considered distributed constrained convex optimization under multiple noise terms in both computation and communication stages.
	By means of an augmented Lagrangian framework, almost sure convergence with a diminishing stepsize policy was established. \cite{nedic2016stochastic} investigated a subgradient-push method for distributed optimization over time-varying directed graphs. For strongly convex objective functions, the method exhibits an $\mathcal{O}(\frac{\ln k}{k})$ convergence rate.
	The work in \cite{sayin2017stochastic} used a time-dependent weighted mixing of stochastic subgradient updates to achieve an $\mathcal{O}(\frac{n\sqrt{n}}{(1-\lambda_2(\mW))k})$ convergence rate for minimizing the sum of  nonsmooth strongly convex functions. 
	\cite{lan2017communication} presented a new class of distributed first-order methods for nonsmooth and stochastic optimization which was shown to exhibit an $\mathcal{O}(\frac{1}{k})$ (respectively, $\mathcal{O}(\frac{1}{\sqrt{k}})$) convergence rate for minimizing the sum of strongly convex functions (respectively, convex functions). 
	The work in \cite{sirb2018decentralized} considered a decentralized algorithm with delayed gradient information which achieves an $\mathcal{O}(\frac{1}{\sqrt{k}})$ rate of convergence for general convex functions. In \cite{jakovetic2018convergence}, an $\mathcal{O}(\frac{1}{k})$ convergence rate was established for strongly convex costs and random networks. Recently, the work in \cite{xin2019variance} proposed a variance-reduced decentralized stochastic optimization method with gradient tracking.
	
	Several recent works 
	%	\cite{chen2012limiting,morral2014success,chen2015learning,chen2015learning2,towfic2016excess,lian2017can,pu2017flocking,morral2017success,pu2018swarming,assran2018stochastic,pu2020distributed,olshevsky2018robust,koloskova2019decentralized} 
	have shown that distributed methods may compare with centralized algorithms under various conditions.
	%	 The paper \cite{morral2014success} first showed that distributed stochastic approximation performs asymptotically as well as centralized schemes by means of a central limit theorem. Under a constant stepsize policy, the work in \cite{chen2015learning,chen2015learning2} first showed that, with a sufficiently small constant stepsize, a distributed stochastic gradient method achieves comparable performance to a centralized method in terms of the steady-state mean-square-error. 
	In addition to \cite{chen2012limiting,chen2015learning,chen2015learning2} discussed before, \cite{morral2014success,morral2017success} proved that distributed stochastic approximation performs asymptotically
	as well as centralized schemes by means of a central limit theorem. \cite{towfic2016excess} first showed that a distributed stochastic
	gradient algorithm asymptotically achieves comparable convergence rate to a centralized method, but assuming that all the local functions $f_i$ have the same minimum. 
	%	 the work in \cite{towfic2016excess} shows that asymptotically, a distributed stochastic gradient algorithm performs as well as the centralized counterpart under a decreasing stepsize policy, but assuming that all the local functions $f_i$ have the same minimum.
	\cite{pu2017flocking,pu2018swarming} demonstrated the advantage of distributively implementing a stochastic gradient method assuming that sampling times are random and non-negligible.  For nonconvex objective functions, \cite{lian2017can} proved that decentralized algorithms can achieve a linear speedup similar to a centralized algorithm when $k$ is large enough. This result was generalized to the setting of directed communication networks in \cite{assran2019stochastic} for training deep neural networks.
	The work in \cite{pu2020distributed} considered a distributed stochastic gradient tracking method which performs as well as centralized stochastic gradient descent under a small enough constant stepsize.
	A recent paper \cite{spiridonoff2020robust} discussed an algorithm that asymptotically performs as well as the best bounds on centralized stochastic gradient descent subject to possible message losses, delays, and asynchrony. In a parallel recent work \cite{koloskova2019decentralized}, a similar result was demonstrated with a further compression technique which allowed nodes to save on communication. For more discussion on the topic of achieving asymptotic network independence in distributed stochastic optimization, the readers are referred to a recent survey \cite{pu2019asymptotic} .
	
	\subsection{Main Contribution}
	We next summarize the main contribution of the paper.
	First, we begin by performing a non-asymptotic convergence analysis of the distributed stochastic gradient descent (DSGD) method.
	For strongly convex and smooth objective functions, DSGD asymptotically achieves the optimal network independent convergence rate compared to centralized stochastic gradient descent (SGD). We explicitly identify the non-asymptotic convergence rate as a function of characteristics of the objective functions and the network. The relevant results are established in Corollary \ref{cor: U(k)_general_step} and Theorem \ref{Thm: best rate}.
	%			 (e.g., spectral gap of the mixing matrix). 
	
	Our main contribution is to characterize the transient time needed for DSGD to approach the asymptotic convergence rate. On the one hand, we show an upper bound of  $K_T=\mathcal{O}\left(\frac{n}{(1-\rho_w)^2}\right)$, where $1-\rho_w$ denotes the spectral gap of the mixing matrix of communicating agents. %This is the best scaling of transient times so far under the considered problem setting, to the best of our knowledge.
	On the other hand, we construct a ``hard" optimization problem for which we show that the transient time needed for DSGD to approach the asymptotic convergence rate is lower bounded by $\Omega\left(\frac{n}{(1-\rho_w)^2}\right)$, implying that this upper bound is sharp.
	
	Additionally, we provide numerical experiments that demonstrate the tightness of the theoretical findings. In particular, for the ring network topology and the square grid network topology, simulations are consistent with the transient time $K_T=\mathcal{O}\left(\frac{n}{(1-\rho_w)^2}\right)$ for solving the \emph{on-line} ridge regression problem.
	
	\subsection{Notation}
	\label{subsec:pre}
	Vectors are column vectors unless otherwise specified.
	Each agent $i$ holds a local copy of the decision vector denoted by $x_i\in\mathbb{R}^p$, and its value at iteration/time $k$ is written as $x_i(k)$. 
	Let
	\begin{equation*}
		\mx := [x_1, x_2, \ldots, x_n]^{\T}\in\mathbb{R}^{n\times p},\quad
		\ox :=  \sp{\frac{1}{n}\sum_{i=1}^n x_i},
	\end{equation*}
	where 
	%	$\mathbf{1}$ is the vector of all ones and 
	$\T$ denotes transpose.
	Define an aggregate objective function
	\begin{equation*}
		F(\mx):=\sum_{i=1}^nf_i(x_i),
	\end{equation*}
	and let
	\begin{equation*}
		\nabla F(\mx):=\left[\nabla f_1(x_1), \nabla f_2(x_2), \ldots, \nabla f_n(x_n)\right]^{\T}\in\mathbb{R}^{n\times p},
	\end{equation*}
	\begin{equation*}
		\bar{\nabla} F(\mx):= \sp{\frac{1}{n}\sum_{i=1}^n \nabla f_i(x_i)}.
	\end{equation*}
	In addition, we denote
	\begin{equation*}
		\boldsymbol{\xi} := [\xi_1, \xi_2, \ldots, \xi_n]^{\T}\in\mathbb{R}^{n\times p},
	\end{equation*}
	\begin{equation*}
		\mg(\mx,\boldsymbol{\xi}):=[g_1(x_1,\xi_1), g_2(x_2,\xi_2), \ldots, g_n(x_n,\xi_n)]^{\T}\in\mathbb{R}^{n\times p}.
	\end{equation*}
	In what follows we write $g_i(k):=g_i(x_i(k),\xi_i(k))$ and $\mg(k):=\mg(\mx(k),\boldsymbol{\xi}(k))$ for short.
	
	The inner product of two vectors $a,b$ is written as $\langle a,b\rangle$. For two matrices $\mA,\mathbf{B}\in\mathbb{R}^{n\times p}$, let
	$\langle \mA,\mathbf{B}\rangle :=\sum_{i=1}^n\langle A_i,B_i\rangle$,
	where $A_i$ (respectively, $B_i$) is the $i$-th row of $\mA$ (respectively, $\mathbf{B}$). We use $\|\cdot\|$ to denote the $2$-norm of vectors and the Frobenius norm of matrices.
	
	A graph $\mathcal{G}=(\mathcal{N},\mathcal{E})$ has a set of vertices (nodes) $\mathcal{N}=\{1,2,\ldots,n\}$ and a set of edges connecting vertices $\mathcal{E}\subseteq \mathcal{N}\times \mathcal{N}$. Suppose agents interact in an undirected graph, i.e., $(i,j)\in\mathcal{E}$ if and only if $(j,i)\in\mathcal{E}$.
	%Each agent $i$ has a set of neighbors $\mathcal{N}_i=\{j\mid j\neq i, (i,j)\in\mathcal{E}\}$. The cardinality of $\mathcal{N}_i$ is referred to as agent $i$'s degree.
	
	Denote the mixing matrix of agents by $\mW=[w_{ij}]\in\mathbb{R}^{n\times n}$. Two agents $i$ and $j$  are connected if and only if $w_{ij}, w_{ji}>0$ ($w_{ij}=w_{ji}=0$ otherwise). Formally, we make the following assumption on the communication among agents.
	\begin{assumption}
		\label{asp: network}
		The graph $\mathcal{G}$ is undirected and connected (there exists a path between any two nodes). The mixing matrix $\mW$ is nonnegative and doubly stochastic, 
		i.e., $\mW\mathbf{1}=\mathbf{1}$ and $\mathbf{1}^{\T}\mW=\mathbf{1}^{\T}$. 
		%{\color{blue}In addition, $w_{ii}>0$ for all $i\in\mathcal{N}$.}
	\end{assumption}
	%In view of this assumption, due to being {\color{blue}symmetric}, the matrix $W$ is doubly stochastic, i.e., $\mathbf{1}^{\T}W=W\mathbf{1}=\mathbf{1}$.
	From Assumption \ref{asp: network}, we have the following contraction property of $\mW$ (see \cite{qu2017harnessing}).
	\begin{lemma}
		\label{lem: spectral norm}
		Let Assumption \ref{asp: network} hold, and let $\rho_w$ denote the spectral norm of 
		the matrix $\mW-\frac{1}{n}\mathbf{1}\mathbf{1}^{\T}$. Then, $\rho_w<1$ and 
		\begin{equation*}
			\|\mW\boldsymbol{\omega}-\mathbf{1}\overline{\omega}\|\le \rho_w\|\boldsymbol{\omega}-\mathbf{1}\overline{\omega}\|
		\end{equation*}
		for all $\boldsymbol{\omega}\in\mathbb{R}^{n\times p}$, where $\overline{\omega} := \frac{1}{n}\mathbf{1}^{\T}\boldsymbol{\omega}$.
	\end{lemma}
	
	The rest of this paper is organized as follows. We present the DSGD algorithm and some preliminary results in Section~\ref{sec: DSGD}. In Section \ref{sec: analysis} we prove the sublinear convergence rate of the algorithm. Our main convergence results and a comparison with the centralized stochastic gradient method are in Section~\ref{sec: main_results}. {\sp Two numerical example are} presented in Section \ref{sec: numerical}, and we conclude the paper in Section \ref{sec: conclusions}.
	
	\section{Distributed Stochastic Gradient Descent}
	\label{sec: DSGD}
	
	We consider the following DSGD method adapted from DGD and the diffusion strategy \cite{chen2012diffusion}: at each step $k\ge 0$, 
	every agent $i$ independently performs the update:
	\begin{equation}
		\label{eq: x_i,k}
		x_i(k+1) = \sum_{j=1}^{n}w_{ij}\left(x_j(k)-\alpha_k g_j(k)\right),
	\end{equation}
	where $\{\alpha_k\}$ is a sequence of non-increasing stepsizes. {\sp The particular choice of the stepsize sequence will be introduced in Section \ref{sec: analysis}.} The {\sp initial vectors $x_{i}(0)$} are arbitrary for all~$i\in{\cal N}$.
	We can rewrite (\ref{eq: x_i,k}) in the following compact form:
	\begin{equation}
		\label{eq: x_k}
		\mx(k+1) = \mW\left(\mx(k)-\alpha_k \mg(k)\right).
	\end{equation}
	
	Throughout the paper, we make the following standing assumption regarding the objective functions $f_i$.\footnote{The assumption can be generalized to the case where the agents have different $\mu$ and $L$.} {\sp These assumptions are satisfied for many machine learning problems, such as linear regression, smooth support vector machine (SVM), logistic regression, and softmax regression.}
	\begin{assumption}
		\label{asp: mu-L_convexity}
		Each $f_i:\mathbb{R}^p\rightarrow \mathbb{R}$ is  $\mu$-strongly convex with $L$-Lipschitz continuous gradients, i.e., for any $x,x'\in\mathbb{R}^p$,
		\begin{equation}
			\begin{split}
				& \langle \nabla f_i(x)-\nabla f_i(x'),x-x'\rangle\ge \mu\|x-x'\|^2,\\
				& \|\nabla f_i(x)-\nabla f_i(x')\|\le L \|x-x'\|.
			\end{split}
		\end{equation}
	\end{assumption}
	Under Assumption \ref{asp: mu-L_convexity}, Problem (\ref{opt Problem_def}) has a unique {\sp optimal solution $x_*\in\mathbb{R}^p$}, and the following result holds (see \cite{qu2017harnessing} Lemma 10).
	\begin{lemma}
		\label{lem: contraction_mu-L_convexity}
		For any $x\in\mathbb{R}^p$ and $\alpha\in(0,2/L)$, we have
		\begin{equation*}
			\|x-\alpha\nabla f(x)-x_*\|\le \lambda\|x-x_*\|,
		\end{equation*}
		where $\lambda=\max(|1-\alpha\mu|,|1-\alpha L|)$.
	\end{lemma}
	
	{\sp Denote $\og(k):= \frac{1}{n}\sum_{i=1}^n g_i(k)$.} The following two results are useful for our analysis.
	\begin{lemma}
		\label{lem: oy_k-h_k}
		Under {\sp Assumptions \ref{asp: gradient samples} and \ref{asp: mu-L_convexity}, 
			for all $k\ge0$,
			\begin{multline*}
				\bE\left[\left\|\og(k)- \bar{\nabla} F(\mx(k))\right\|^2\mid \mF(k)\right] \\
				\le \frac{2ML^2}{n^2}\|\mx(k)-\mathbf{1}x_*^{\T}\|^2+\frac{\bar{M}}{n},
			\end{multline*}
			where
			\begin{equation*}
				\bar{M} := \frac{2M\sum_{i=1}^n\|\nabla f_i(x_*)\|^2}{n}+\sigma^2.
			\end{equation*}
		}
	\end{lemma}
	\begin{proof}
		By definitions of $\og(k)$, $\bar{\nabla} F(\mx(k))$ and Assumption \ref{asp: gradient samples}, we have
		\begin{align*}
			&\bE\left[\left\|\og(k)- \bar{\nabla} F(\mx(k))\right\|^2\right]\\
			& \quad = \sp{\bE\left[\left\|\frac{1}{n}\sum_{i=1}^n g_i(k)- \frac{1}{n}\sum_{i=1}^n \nabla f_i(x_i(k))\right\|^2\right]}\\
			& \quad = \frac{1}{n^2}\sum_{i=1}^n\bE\left[\|g_i(k)-\nabla f_i(x_i(k))\|^2\right] \\
			& \quad \le \frac{\sigma^2}{n}{\sp +\frac{M\sum_{i=1}^n \|\nabla f_i(x_i(k))\|^2}{n^2}}.
		\end{align*}
		{\sp Notice that $\|\nabla f_i(x_i(k))\|^2=\|\nabla f_i(x_i(k))-\nabla f_i(x_*)+\nabla f_i(x_*)\|^2\le 2\|\nabla f_i(x_i(k))-\nabla f_i(x_*)\|^2+2\|\nabla f_i(x_*)\|^2\le 2L^2\| x_i(k)- x_*\|^2+2\|\nabla f_i(x_*)\|^2$ from Assumption \ref{asp: mu-L_convexity}. We have
			\begin{multline*}
				\bE\left[\left\|\og(k)- \bar{\nabla} F(\mx(k))\right\|^2\right] \le \frac{2ML^2}{n^2}\|\mx(k)-\mathbf{1}x_*^{\T}\|^2\\
				+\frac{1}{n}\left(\frac{2M\sum_{i=1}^n\|\nabla f_i(x_*)\|^2}{n}+\sigma^2\right).
			\end{multline*}
		}
	\end{proof}
	\begin{lemma}
		\label{lem: strong_convexity}
		Under Assumption \ref{asp: mu-L_convexity}, for all $k\ge0$,
		\begin{align}
			\left\| \nabla f(\ox(k))-\bar{\nabla} F(\mx(k))\right\| \le \frac{L}{\sqrt{n}}\|\mx(k)-\mathbf{1}\ox(k)^{\T}\|.
		\end{align}
	\end{lemma}
	\begin{proof}
		By definition,
		\begin{align*}
			& \left\| \nabla f(\ox(k))-\bar{\nabla} F(\mx(k))\right\| \\
			& \quad =  \left\|\frac{1}{n}\sum_{i=1}^n\nabla f_i(\ox(k))-\frac{1}{n}\sum_{i=1}^n\nabla f_i(x_i(k))\right\|\\
			\text{(Assumption \ref{asp: mu-L_convexity})} & \quad \le \frac{L}{n}\sum_{i=1}^n\|\ox(k)-x_i(k)\|\\
			& \quad  \le \frac{L}{\sqrt{n}}\|\mx(k)-\mathbf{1}\ox(k)^{\T}\|,
		\end{align*}
		where the last relation follows from {\sp H\"{o}lder's} inequality.
	\end{proof}

	\subsection{Preliminary Results}
	\label{subsec: pre_results}
	
	In this section, we present some preliminary results concerning $\bE[\|\ox(k)-x_*\|^2]$ (expected optimization error) and $\bE[\|\mx(k)-\mathbf{1}\ox(k)^{\T}\|^2]$ (expected consensus error). Specifically, we bound the two terms by linear combinations of their values in the last iteration.
	%	Throughout the analysis we assume Assumptions \ref{asp: gradient samples}, \ref{asp: network} and \ref{asp: mu-L_convexity} hold.
	
	{\sp For ease of presentation,} for all $k$ we denote
	\begin{equation*}
		U(k):=\bE\left[\|\ox(k)-x_*\|^2\right],\, V(k):=\bE[\|\mx(k)-\mathbf{1}\ox(k)^{\T}\|^2].
	\end{equation*}
	{\sp In the lemma below, we bound the optimization error $U(k+1)$ by several error terms at iteration $k$, including the consensus error $V(k)$. It serves as a starting point for the follow-up analysis.}
	\begin{lemma}
		\label{lem: optimization_error_contraction_pre}
		{\sp Suppose Assumptions \ref{asp: gradient samples}-\ref{asp: mu-L_convexity} hold.} Under Algorithm \eqref{eq: x_k}, {\sp supposing $\alpha_k\le \frac{1}{L}$, we have
			\begin{align}
				\label{Opt_error_pre}
				& U(k+1)	\le (1-\alpha_k\mu)^2 U(k) \notag\\
				& \quad +\frac{2\alpha_k L}{\sqrt{n}}(1-\alpha_k\mu)\bE[\|\ox(k)-x_*\|\|\mx(k)-\mathbf{1}\ox(k)^{\T}\|] \notag\\
				& \quad +\frac{\alpha_k^2L^2}{n}V(k) +\alpha_k^2\sp{\left(\frac{2ML^2}{n^2}\bE[\|\mx(k)-\mathbf{1}x_*^{\T}\|^2]+\frac{\bar{M}}{n}\right)}.
		\end{align}}
	\end{lemma}
	\begin{proof}
		By the definitions of $\bar{x}(k)$, $\og(k)$ and relation \eqref{eq: x_k}, we have
		$
		\ox(k+1)=\ox(k)-\alpha_k\og(k).
		$
		Hence,
		\begin{align*}
			\begin{array}{ll}
				& \|\ox(k+1)-x_*\|^2=\left\|\ox(k)-\alpha_k\og(k)-x_*\right\|^2\\
				= & \left\|\ox(k)-\alpha_k \bar{\nabla} F(\mx(k))-x_*+\alpha_k \bar{\nabla} F(\mx(k))-\alpha_k\og(k)\right\|^2\\
				= & \|\ox(k)-\alpha_k \bar{\nabla} F(\mx(k))-x_*\|^2+ \alpha_k^2 \|\bar{\nabla} F(\mx(k))-\og(k)\|^2\\
				& + 2\alpha_k\langle \ox(k)-\alpha_k \bar{\nabla} F(\mx(k))-x_*, \bar{\nabla} F(\mx(k))-\og(k)\rangle.
				% =\left\|\ox(k)-x_*\right\|^2-2\alpha_k\langle \ox(k)-x_*,\og(k)\rangle+\alpha_k^2\left\|\og(k)\right\|^2.
			\end{array}
		\end{align*}
		Noting that 
		$\bE[\og(k)\mid\mx(k)]=\bar{\nabla} F(\mx(k))$ and $\bE[\left\|\og(k)- \bar{\nabla} F(\mx(k))\right\|^2\mid \mF(k)] \le \sp{\frac{2ML^2}{n^2}\|\mx(k)-\mathbf{1}x_*^{\T}\|^2+\frac{\bar{M}}{n}}$ from Lemma \ref{lem: oy_k-h_k}, we obtain
		\begin{multline}
			\label{optimization_error_contraction_pre_relation1}
			\bE[\|\ox(k+1)-x_*\|^2\mid \mF(k)]
			\le\|\ox(k)-\alpha_k \bar{\nabla} F(\mx(k))-x_*\|^2\\
			+\alpha_k^2\left(\sp{\frac{2ML^2}{n^2}\|\mx(k)-\mathbf{1}x_*^{\T}\|^2+\frac{\bar{M}}{n}}\right).
			% =\left\|\ox(k)-x_*\right\|^2-2\alpha_k\langle \ox(k)-x_*,\og(k)\rangle+\alpha_k^2\left\|\og(k)\right\|^2.
		\end{multline}
		We next bound the first term on the right-hand-side of \eqref{optimization_error_contraction_pre_relation1}.
		\begin{align*}
			& \|\ox(k)-\alpha_k \bar{\nabla} F(\mx(k))-x_*\|^2\\
			= & \|\ox(k)-\alpha_k\nabla f(\ox(k))-x_*+\alpha_k\nabla f(\ox(k))-\alpha_k \bar{\nabla} F(\mx(k))\|^2\\
			\le & \|\ox(k)-\alpha_k\nabla f(\ox(k))-x_*\|^2\\
			& +2\alpha_k\|\ox(k)-\alpha_k\nabla f(\ox(k))-x_*\|\|\nabla f(\ox(k))-\bar{\nabla} F(\mx(k))\| \\
			& +\alpha_k^2\|\nabla f(\ox(k))-\bar{\nabla} F(\mx(k))\|^2,
		\end{align*}
		where we used the Cauchy-Schwarz inequality.
		By Lemma \ref{lem: oy_k-h_k}, 
		\begin{equation*}
			\begin{array}{l}
				\|\nabla f(\ox(k))-\bar{\nabla} F(\mx(k))\|^2\le \frac{L^2}{n}\|\mx(k)-\mathbf{1}\ox(k)^{\T}\|^2.
			\end{array}
		\end{equation*}
		{\sp Since $\alpha_k\le \frac{1}{L}$, in light of Lemma \ref{lem: contraction_mu-L_convexity},
			\begin{equation*}
				\begin{array}{l}
					\left\|\ox(k)-\alpha_k\nabla f(\ox(k))-x_*\right\|^2\le (1-\alpha_k\mu)^2\left\|\ox(k)-x_*\right\|^2.
				\end{array}
		\end{equation*}}
		{\sp Then we have 
			\begin{multline}
				\label{optimization_error_contraction_pre_relation2}
				\|\ox(k)-\alpha_k \bar{\nabla} F(\mx(k))-x_*\|^2\le (1-\alpha_k\mu)^2\|\ox(k) - x_*\|^2\\
				+\frac{2\alpha_k L}{\sqrt{n}}(1-\alpha_k\mu)\|\ox(k) - x_*\|\|\mx(k)-\mathbf{1}\ox(k)^{\T}\|\\
				+\frac{\alpha_k^2L^2}{n}\|\mx(k)-\mathbf{1}\ox(k)^{\T}\|^2.
			\end{multline}
			In light of relation \eqref{optimization_error_contraction_pre_relation2}, taking full expectation on both sides of relation \eqref{optimization_error_contraction_pre_relation1} yields the result.
		}
		%		See Appendix \ref{proof lem: optimization_error_contraction_pre}.
	\end{proof}
	
	The next result is a corollary of Lemma \ref{lem: optimization_error_contraction_pre} {\sp with an additional condition on the stepsize $\alpha_k$. We are able to remove the cross term in relation \eqref{Opt_error_pre} and obtain a cleaner expression, which facilitates our later analysis.}
	\begin{lemma}
		\label{lem: optimization_error_contraction}
		{\sp Suppose Assumptions \ref{asp: gradient samples}-\ref{asp: mu-L_convexity} hold.} Under Algorithm \eqref{eq: x_k}, supposing $\alpha_k\le \min\{\frac{1}{L},\frac{1}{3\mu}\}$, then
		\begin{multline}
			\label{Opt_contraction_k>=K1}
			U(k+1)
			\le \left(1-\frac{3}{2}\alpha_k\mu\right)U(k)
			+\frac{3\alpha_k L^2}{n\mu}V(k)\\
			+\alpha_k^2\left(\sp{\frac{2ML^2}{n^2}\bE[\|\mx(k)-\mathbf{1}x_*^{\T}\|^2]+\frac{\bar{M}}{n}}\right).
		\end{multline}
	\end{lemma}
	\begin{proof}
		{\sp From Lemma \ref{lem: optimization_error_contraction_pre},
			\begin{align*}
				& U(k+1) \le (1-\alpha_k\mu)^2 U(k)+(1-\alpha_k\mu)^2 c U(k)\\
				& \quad +\frac{\alpha_k^2L^2}{n}\frac{1}{c}V(k) +\frac{\alpha_k^2L^2}{n}V(k)\\
				& \quad +\alpha_k^2\left(\sp{\frac{2ML^2}{n^2}\bE[\|\mx(k)-\mathbf{1}x_*^{\T}\|^2]+\frac{\bar{M}}{n}}\right)\\
				& \le (1+c)(1-\alpha_k\mu)^2 U(k) +\left(1+\frac{1}{c}\right)\frac{\alpha_k^2L^2}{n}V(k)\\
				& \quad +\alpha_k^2\left(\sp{\frac{2ML^2}{n^2}\bE[\|\mx(k)-\mathbf{1}x_*^{\T}\|^2]+\frac{\bar{M}}{n}}\right),
			\end{align*}
			where $c>0$ is arbitrary. }
		
		Take $c=\frac{3}{8}\alpha_k\mu$.
		Noting that $\alpha_k\le \frac{1}{3\mu}$, we have $(1+c)(1-\alpha_k\mu)^2\le 1-\frac{3}{2}\alpha_k\mu$, and $(1+\frac{1}{c})\alpha_k\le \frac{3}{\mu}$.
		{\sp Thus,
			\begin{multline*}
				U(k+1)
				\le  \left(1-\frac{3}{2}\alpha_k\mu\right)U(k)
				+\frac{3\alpha_k L^2}{n\mu}V(k)\\
				+\alpha_k^2\left(\frac{2ML^2}{n^2}\bE[\|\mx(k)-\mathbf{1}x_*^{\T}\|^2]+\frac{\bar{M}}{n}\right).
		\end{multline*}}
		
		%		See Appendix \ref{proof lem: optimization_error_contraction}.
	\end{proof}
	
	{\sp Since the consensus error term $V(k)$ plays a key role in the statements of Lemma \ref{lem: optimization_error_contraction_pre} and Lemma \ref{lem: optimization_error_contraction}, we present the following lemma that bounds $V(k+1)$. }
	\begin{lemma}
		\label{lem: consensus_error_contraction}
		{\sp Suppose Assumptions \ref{asp: gradient samples}-\ref{asp: mu-L_convexity} hold. Under Algorithm \eqref{eq: x_k}, for all $k\ge 0$,
			\begin{align}
				\label{V(k+1) bound_pre}
				& V(k+1)\le \frac{\left(3+\rho_w^2\right)}{4}V(k)+\alpha_k^2\rho_w^2 n\sigma^2 \notag\\
				& \quad +2\alpha_k^2\rho_w^2\left(\frac{3}{1-\rho_w^2}+M\right)\left(L^2\bE[\|\mx(k)-\mathbf{1}x_*^{\T}\|^2] \right. \notag\\
				& \quad \left.+ \|\nabla F(\mathbf{1}x_*^{\T})\|^2\right).
		\end{align}}
		%		\begin{multline*}
		%		V(k+1)
		%		\le \rho_w^2\left(\frac{3-\rho_w^2}{2}+2\alpha_k L+2\alpha_k^2{\sp(1+M)} L^2\right)V(k)\\
		%		+\rho_w^2\alpha_k^2\left({\sp c_0}nL^2 U(k)+{\sp c_0}\|\nabla F(\mathbf{1}x_*^{\T})\|^2+n\sigma^2\right),
		%		\end{multline*}
		%	where 
		%	\begin{equation*}
		%		c_0 := {\sp 4M}+\frac{8}{(1-\rho_w^2)}.
		%	\end{equation*}
	\end{lemma}
	\begin{proof}
		From relation \eqref{eq: x_k}, 
		\begin{align*}
			& \mx(k+1)-\mathbf{1}\ox(k+1)^{\T} \\
			& = \mW\left(\mx(k)-\alpha_k \mg(k)\right)-\mathbf{1}(\ox(k)-\alpha_k\og(k))\\
			& = \left(\mW-\frac{\mathbf{1}\mathbf{1}^{\T}}{n}\right)\left[(\mx(k)-\mathbf{1}\ox(k)^{\T})-\alpha_k(\mg(k)-\mathbf{1}\og(k)^{\T})\right],
		\end{align*}
		we have
		\begin{align*}
			&\|\mx(k+1)-\mathbf{1}\ox(k+1)^{\T}\|^2\\
			& \le \rho_w^2\left\|(\mx(k)-\mathbf{1}\ox(k)^{\T})-\alpha_k(\mg(k)-\mathbf{1}\og(k)^{\T})\right\|^2\\
			& = \rho_w^2\left[\left\|\mx(k)-\mathbf{1}\ox(k)^{\T}\right\|^2+\alpha_k^2\|\mg(k)-\mathbf{1}\og(k)^{\T}\|^2\right.\\
			& \quad -2\alpha_k\langle \mx(k)-\mathbf{1}\ox(k)^{\T},\mg(k)-\mathbf{1}\og(k)^{\T}\rangle\Big ].
			%\\
			%\le \rho_w^2\left\|\mx(k)-\mathbf{1}\ox(k)^{\T}\right\|^2+\alpha_k^2\|\mg(k)-\mathbf{1}\og(k)^{\T}\|^2
			%-2\alpha_k\langle\left(\mW-\frac{\mathbf{1}\mathbf{1}^{\T}}{n}\right)(\mx(k)-\mathbf{1}\ox(k)^{\T}),\mg(k)-\mathbf{1}\og(k)^{\T}\rangle,
		\end{align*}
		%	{\sp(Lemma...)}
		Since $\bE[\mg(k)\mid \mF(k)]=\nabla F(\mx(k))$ and $\bE[\og(k)\mid \mF(k)]=\bar{\nabla} F(\mx(k))$,
		\begin{multline*}
			\bE\left[\langle \mx(k)-\mathbf{1}\ox(k)^{\T},\mg(k)-\mathbf{1}\og(k)^{\T}\rangle\mid \mF(k)\right]\\
			=\langle \mx(k)-\mathbf{1}\ox(k)^{\T},\nabla F(\mx(k))-\mathbf{1}\bar{\nabla} F(\mx(k))\rangle,
		\end{multline*}
		and
		\begin{align*}
			& \bE[\|\mg(k)-\mathbf{1}\og(k)^{\T}\|^2\mid \mF(k)]\\
			& = \bE[\|\nabla F(\mx(k))-\mathbf{1}\bar{\nabla} F(\mx(k))-\nabla F(\mx(k))+\mathbf{1}\bar{\nabla} F(\mx(k))\\
			& \quad +\mg(k)-\mathbf{1}\og(k)^{\T}\|^2\mid \mF(k)]\\
			& = \|\nabla F(\mx(k))-\mathbf{1}\bar{\nabla} F(\mx(k))\|^2 \\
			& \quad +\bE[\|\nabla F(\mx(k))-\mg(k)- \\
			& \quad \quad \quad \quad (\mathbf{1}\bar{\nabla} F(\mx(k))-\mathbf{1}\og(k)^{\T})\|^2\mid \mF(k)]\\
			& \le \|\nabla F(\mx(k))-\mathbf{1}\bar{\nabla} F(\mx(k))\|^2\\
			& \quad+\bE[\|\nabla F(\mx(k))-\mg(k)\|^2\mid \mF(k)]\\
			& \le \|\nabla F(\mx(k))-\mathbf{1}\bar{\nabla} F(\mx(k))\|^2+n\sigma^2+\sp{M\|\nabla F(\mx(k))\|^2},
		\end{align*}
		where the last inequality follows from Assumption \ref{asp: gradient samples}.
		Therefore (assuming $\rho_w> 0$),
		\begin{align*}
			& \frac{1}{\rho_w^2}\bE[\|\mx(k+1)-\mathbf{1}\ox(k+1)^{\T}\|^2\mid \mF(k)]\\
			& \le \left\|\mx(k)-\mathbf{1}\ox(k)^{\T}\right\|^2+\alpha_k^2\|\nabla F(\mx(k))-\mathbf{1}\bar{\nabla} F(\mx(k))\|^2\\
			& \quad +\alpha_k^2\left(n\sigma^2+\sp{M\|\nabla F(\mx(k))\|^2}\right)\\
			& \quad -2\alpha_k\langle \mx(k)-\mathbf{1}\ox(k)^{\T},\nabla F(\mx(k))-\mathbf{1}\bar{\nabla} F(\mx(k))\rangle\\
			& \le \left\|\mx(k)-\mathbf{1}\ox(k)^{\T}\right\|^2+\alpha_k^2{\sp(1+M)}\|\nabla F(\mx(k))\|^2+\alpha_k^2n\sigma^2\\
			& \quad+2\alpha_k\left\|\mx(k)-\mathbf{1}\ox(k)^{\T}\right\|\left\|\nabla F(\mx(k))\right\|\\
			& {\sp\le (1+c)\left\|\mx(k)-\mathbf{1}\ox(k)^{\T}\right\|^2+\alpha_k^2(1+M)\|\nabla F(\mx(k))\|^2}\\
			& \quad {\sp+\alpha_k^2n\sigma^2+\frac{\alpha_k^2}{c}\left\|\nabla F(\mx(k))\right\|^2},
		\end{align*}
		where $c>0$ is arbitrary.
		{\sp Letting $c=\frac{1-\rho_w^2}{2}$ and noting that by Assumption \ref{asp: mu-L_convexity},
			\begin{align*}
				& \|\nabla F(\mx(k))\|^2\le 2\|\nabla F(\mx(k))-\nabla F(\mathbf{1}x_*^{\T})\|^2+2\|\nabla F(\mathbf{1}x_*^{\T})\|^2\\
				& \le 2L^2\|\mx(k)-\mathbf{1}x_*^{\T}\|^2 + 2\|\nabla F(\mathbf{1}x_*^{\T})\|^2,
			\end{align*}
			we have
			\begin{align*}
				& \frac{1}{\rho_w^2}\bE[\|\mx(k+1)-\mathbf{1}\ox(k+1)^{\T}\|^2\mid \mF(k)]-\alpha_k^2n\sigma^2\\
				& \le \left(\frac{3-\rho_w^2}{2}\right)\left\|\mx(k)-\mathbf{1}\ox(k)^{\T}\right\|^2\\
				& \quad +2\alpha_k^2\left(\frac{3}{1-\rho_w^2}+M\right)(L^2\|\mx(k)-\mathbf{1}x_*^{\T}\|^2 + \|\nabla F(\mathbf{1}x_*^{\T})\|^2).
			\end{align*}
			Notice that $\rho_w^2\left(\frac{3-\rho_w^2}{2}\right)\le \frac{(3+\rho_w^2)}{4}$. In light of Lemma \ref{lem: bounded_iterates_estimate_general_step}, taking full expectation on both sides of the above inequality leads to
			\begin{align*}
				& \bE[\|\mx(k+1)-\mathbf{1}\ox(k+1)^{\T}\|^2]-\alpha_k^2\rho_w^2 n\sigma^2\\
				& \le \frac{(3+\rho_w^2)}{4}\bE[\left\|\mx(k)-\mathbf{1}\ox(k)^{\T}\right\|^2]\\
				& \quad +2\alpha_k^2\rho_w^2\left(\frac{3}{1-\rho_w^2}+M\right)\left(L^2\bE[\|\mx(k)-\mathbf{1}x_*^{\T}\|^2] \right.\\
				& \quad \left.+ \|\nabla F(\mathbf{1}x_*^{\T})\|^2\right).
			\end{align*}
		}
		%		See Appendix \ref{proof lem: consensus_error_contraction}.
	\end{proof}
	
	\section{Analysis}
	\label{sec: analysis}
	
	We are now ready to derive some preliminary convergence results for Algorithm \eqref{eq: x_k}. First, we provide a uniform bound on the iterates generated by Algorithm \eqref{eq: x_k} (in expectation) for all $k\ge 0$.
	Then based on the lemmas established in Section~\ref{subsec: pre_results}, we prove the {\sp sublinear convergence rate for Algorithm \eqref{eq: x_k}, i.e., $U(k)=\mathcal{O}(\frac{1}{k})$ and $V(k)=\mathcal{O}(\frac{1}{k^2})$. These results provide the foundation for our main convergence theorems in Section~\ref{sec: main_results}.}
	
	From now on we consider the following stepsize policy:
	\begin{equation}
		\label{Stepsize}
		\alpha_k:=\frac{\theta}{\mu (k+K)}, \quad \forall k,
	\end{equation}
	where {\sp constant} $\theta>1$, and
	\begin{equation}
		\label{def: K}
		K:= \left\lceil\frac{2\theta {\sp(1+M)}L^2}{\mu^2}\right\rceil,
	\end{equation}
	{\sp with $\lceil\cdot\rceil$ denoting the ceiling function.}
	
	\subsection{Uniform Bound}
	\label{subsec: uniform_bound}
	
	%Then $\alpha_k\le \frac{\mu}{2L^2}$ for all $k\ge K$.
	We first derive a uniform bound on the iterates generated by Algorithm \eqref{eq: x_k} (in expectation) for all $k\ge 0$. {\sp Such a result is helpful for bounding the error terms on the right hand sides of \eqref{Opt_error_pre}, \eqref{Opt_contraction_k>=K1} and \eqref{V(k+1) bound_pre}.}
	\begin{lemma}
		\label{lem: bounded_iterates_estimate_general_step}
		{\sp Suppose Assumptions \ref{asp: gradient samples}-\ref{asp: mu-L_convexity} hold.} Under Algorithm \eqref{eq: x_k} {\sp with stepsize policy \eqref{Stepsize}}, for all $k\ge 0$, we have
		\begin{multline}
			\label{iterates_uniform_bound_from_opt}
			\bE[\|\mx(k)-\mathbf{1}x_*^{\T}\|^2]\le \hat{X}\\
			:=\max\left\{\|\mx(0)-\mathbf{1}x_*^{\T}\|^2,\,\frac{9\|\nabla F(\mathbf{1}x_*^{\T})\|^2}{\mu^2}+\frac{n\sigma^2}{{\sp(1+M)}L^2}\right\}.
		\end{multline}
	\end{lemma}
	\begin{proof}
		The following arguments are inspired by those in \cite{nedic2016stochastic}. 
		
		{\sp First, we bound $\bE\left[\|x_i(k)-\alpha_k g_i(k)\|^2\right]$ for all $i\in\mathcal{N}$ and $k\ge 0$.} By Assumption \ref{asp: gradient samples},
		\begin{align*}
			& \bE\left[\|x_i(k)-\alpha_k g_i(k)\|^2\mid \mF(k)\right]\\
			& = \|x_i(k)-\alpha_k \nabla f_i(x_i(k))\|^2\\
			& \quad +\alpha_k^2\bE\left[\|\nabla f_i(x_i(k))-g_i(k)\|^2\mid \mF(k)\right]\\
			& \le \|x_i(k)\|^2-2\alpha_k \langle\nabla f_i(x_i(k)), x_i(k) \rangle+\alpha_k^2\|\nabla f_i(x_i(k))\|^2\\
			& \quad+\alpha_k^2(\sigma^2{\sp + M\|\nabla f_i(x_i(k))\|^2}).
		\end{align*}
		From the strong convexity and Lipschitz continuity of $f_i$, we know that
		\begin{align*}
			&\langle\nabla f_i(x_i(k)), x_i(k) \rangle\\
			&=  \langle\nabla f_i(x_i(k))-\nabla f_i(0), x_i(k)-0 \rangle+\langle\nabla f_i(0), x_i(k) \rangle\\
			& \ge  \mu\|x_i(k)\|^2+\langle\nabla f_i(0), x_i(k) \rangle,
		\end{align*}
		and
		\begin{multline*}
			\|\nabla f_i(x_i(k))\|^2=\|\nabla f_i(x_i(k))-\nabla f_i(0)+\nabla f_i(0)\|^2\\
			\le 2L^2\|x_i(k)\|^2+2\|\nabla f_i(0)\|^2.
		\end{multline*}
		Hence,
		\begin{align*}
			\begin{split}
				&\bE\left[\|x_i(k)-\alpha_k g_i(k)\|^2\mid \mF(k)\right]\\
				&\le \|x_i(k)\|^2-2\alpha_k \left[\mu\|x_i(k)\|^2+\langle\nabla f_i(0), x_i(k) \rangle\right]\\
				&\quad +2\alpha_k^2 {\sp(1+M)}\left(L^2\|x_i(k)\|^2+\|\nabla f_i(0)\|^2\right)+\alpha_k^2\sigma^2\\
				&\le \|x_i(k)\|^2 - 2\alpha_k\mu\|x_i(k)\|^2   +  2\alpha_k \|\nabla f_i(0)\|\|x_i(k)\|\\
				&\quad + 2\alpha_k^2 {\sp(1+M)}\left(L^2\|x_i(k)\|^2 + \|\nabla f_i(0)\|^2\right)+ \alpha_k^2\sigma^2\\
				&\le [1-2\alpha_k\mu+2\alpha_k^2{\sp(1+M)} L^2]\|x_i(k)\|^2\\
				&\quad +2\alpha_k \|\nabla f_i(0)\|\|x_i(k)\|+\alpha_k^2\left[2{\sp(1+M)}\|\nabla f_i(0)\|^2+\sigma^2\right].
			\end{split}
		\end{align*}
		{\sp Taking full expectation on both sides,} it follows that
		\begin{align*}
			\begin{split}
				& \bE\left[\|x_i(k)-\alpha_k g_i(k)\|^2\right]\\
				& \le [1-2\alpha_k\mu+2\alpha_k^2{\sp(1+M)} L^2]\bE[\|x_i(k)\|^2]\\
				& \quad +2\alpha_k \|\nabla f_i(0)\|\sqrt{\bE[\|x_i(k)\|^2]}\\
				& \quad +\alpha_k^2\left[2{\sp(1+M)}\|\nabla f_i(0)\|^2+\sigma^2\right].
			\end{split}
		\end{align*}
		From the definition of $K$ {\sp in \eqref{def: K}}, $\alpha_k\le \frac{\mu}{2{\sp(1+M)}L^2}$ for all $k\ge 0$. Hence,
		\begin{equation}
			\label{Inequality: x-alpha g}
			\begin{split}
				& \bE\left[\|x_i(k)-\alpha_k g_i(k)\|^2\right]\\
				&\le  (1-\alpha_k\mu)\bE[\|x_i(k)\|^2]+2\alpha_k \|\nabla f_i(0)\|\sqrt{\bE[\|x_i(k)\|^2]}\\
				&\quad+\alpha_k^2\left[2{\sp(1+M)}\|\nabla f_i(0)\|^2+\sigma^2\right]\\
				&\le   \bE[\|x_i(k)\|^2]-\alpha_k\Big[\mu\bE[\|x_i(k)\|^2]-2\|\nabla f_i(0)\|\sqrt{\bE[\|x_i(k)\|^2]}\\
				& \quad\left.-\frac{\mu}{2L^2}\left(2\|\nabla f_i(0)\|^2+\frac{\sigma^2}{{\sp(1+M)}}\right)\right].
			\end{split}
		\end{equation}
		Let us define the following set:
		\begin{multline}
			\label{definition: X}
			\mathcal{X}_i:= \Big\{q\ge 0: \mu q-2\|\nabla f_i(0)\|\sqrt{q}\\
			\left.-\frac{\mu}{2L^2}\left(2\|\nabla f_i(0)\|^2+\frac{\sigma^2}{{\sp(1+M)}}\right)\le 0\right\},
		\end{multline}
		which is non-empty and compact. If $\bE[\|x_i(k)\|^2]\notin \mathcal{X}_i$, we know from inequality (\ref{Inequality: x-alpha g}) that $\bE\left[\|x_i(k)-\alpha_k g_i(k)\|^2\right]\le \bE[\|x_i(k)\|^2]$.
		Otherwise,
		\begin{equation*}
			\begin{split}
				& \bE\left[\|x_i(k)-\alpha_k g_i(k)\|^2\right]
				\le \max_{q\in\mathcal{X}_i}\left\{q-\frac{\mu}{2{\sp(1+M)}L^2}\Big[\mu q\right.\\
				&\quad \left.\left. -2\|\nabla f_i(0)\|\sqrt{q}-\frac{\mu}{2L^2}\left(2\|\nabla f_i(0)\|^2+\frac{\sigma^2}{{\sp(1+M)}}\right)\right]\right\}\\
				&=  \max_{q\in\mathcal{X}_i}\left\{\left(1-\frac{\mu^2}{2{\sp(1+M)}L^2}\right)q+\frac{\mu}{{\sp(1+M)}L^2}\|\nabla f_i(0)\|\sqrt{q}\right.\\
				&\quad \left.+\frac{\mu^2}{4{\sp(1+M)}L^4}\left(2\|\nabla f_i(0)\|^2+\frac{\sigma^2}{{\sp(1+M)}}\right)\right\}.
			\end{split}
		\end{equation*}
		Define the last term above as $R_i$. The previous arguments imply that for all $k\ge 0$,
		\begin{equation*}
			\sp{\bE\left[\|x_i(k)-\alpha_k g_i(k)\|^2\right]\le \max\{\bE\left[\|x_i(k)\|^2\right],R_i\}.}
		\end{equation*}	
		Note that {\sp from relation \eqref{eq: x_k},}
		\begin{multline*}
			\|\mx(k+1)\|^2\le \|\mW\|^2\|\mx(k)-\alpha_k\mg(k)\|^2\\
			\le \|\mx(k)-\alpha_k\mg(k)\|^2.
		\end{multline*}
		We have
		\begin{equation}
			\label{iterates_uniform_bound_from_origin}
			\begin{array}{l}
				\bE[\|\mx(k)\|^2]\le \max\left\{\|\mx(0)\|^2,\sum_{i=1}^n R_i\right\}. 
			\end{array}
		\end{equation}
		
		We further bound $R_i$ as follows. From the definition of $\mathcal{X}_i$,
		\[\max_{q\in\mathcal{X}_i}q \le \linebreak[4] \frac{8\|\nabla f_i(0)\|^2}{\mu^2} + \frac{3\sigma^2}{4{\sp(1+M)}L^2}.\]
		Hence,
		\begin{equation}
			\label{bound: R_i}
			\begin{array}{l}
				R_i =  \max_{q\in\mathcal{X}_i}\left\{q-\frac{\mu}{2{\sp(1+M)}L^2}\Big[\mu q-2\|\nabla f_i(0)\|\sqrt{q}\right.\\
				\quad\left.\left.-\frac{\mu}{2L^2}\left(2\|\nabla f_i(0)\|^2+\frac{\sigma^2}{{\sp(1+M)}}\right)\right]\right\}\\
				\le \max_{q\in\mathcal{X}_i}q-\frac{\mu}{2{\sp(1+M)}L^2}\min_{q\in\mathcal{X}_i}\Big\{\mu q-2\|\nabla f_i(0)\|\sqrt{q}\\
				\quad\left.-\frac{\mu}{2L^2}\left(2\|\nabla f_i(0)\|^2+\frac{\sigma^2}{{\sp(1+M)}}\right)\right\}\\
				\le \frac{8\|\nabla f_i(0)\|^2}{\mu^2}+\frac{3\sigma^2}{4{\sp(1+M)}L^2}\\
				\quad+\frac{\mu}{2{\sp(1+M)}L^2}\left[\frac{\|\nabla f_i(0)\|^2}{\mu}+\frac{\mu}{2L^2}\left(2\|\nabla f_i(0)\|^2+\frac{\sigma^2}{{\sp(1+M)}}\right)\right]\\
				\le \frac{9\|\nabla f_i(0)\|^2}{\mu^2}+\frac{\sigma^2}{{\sp(1+M)}L^2}.
			\end{array}
		\end{equation}
		In light of inequality (\ref{bound: R_i}), further noticing that the choice of $0$ is arbitrary in the proof of (\ref{iterates_uniform_bound_from_origin}), we obtain the uniform bound for $\bE[\|\mx(k)-\mathbf{1}x_*^{\T}\|^2]$ in (\ref{iterates_uniform_bound_from_opt}).	
	\end{proof}
	{\sp The uniform bound provided in Lemma \ref{lem: bounded_iterates_estimate_general_step} is critical for deriving the sublinear convergence rates of $U(k)$ and $V(k$), as it holds for all $k\ge 0$.}
	
	\subsection{Sublinear Rate}
	\label{subsec: sublinear_rate}
	
	%The two terms above represent the expected optimization error and consensus error at iteration $k$, respectively.
	{\sp With the help of} Lemma \ref{lem: optimization_error_contraction} and Lemma \ref{lem: consensus_error_contraction} from Section \ref{subsec: pre_results} {\sp and Lemma \ref{lem: bounded_iterates_estimate_general_step}}, we show in Lemma \ref{lemma: prelim_rates_general_stepsize} below that Algorithm \eqref{eq: x_k} enjoys the sublinear convergence rate, i.e., $U(k)=\mathcal{O}(\frac{1}{k})$ and $V(k)=\mathcal{O}(\frac{1}{k^2})$. 
	%	The proof utilizes the following Lyapunov function:
	%	\begin{equation}
	%		W(k):=U(k)+\omega(k)V(k), \quad \forall k,
	%	\end{equation}
	%	where $\omega(k)> 0$ is to be determined later. 
	For the ease of analysis, we define two auxiliary variables:
	\begin{multline}
		\label{def: tUtV}
		\tilde{U}(k):=U(k-K),\, \tilde{V}(k):=V(k-K),\,
		\forall k\ge K.
	\end{multline}
	{\sp We first derive uniform upper bounds for $U(k)$ and $V(k)$ respectively for all $k\ge 0$ based on Lemma \ref{lem: bounded_iterates_estimate_general_step}. With these bounds, we are able to characterize the constants appearing in the sublinear convergence rates for $U(k)$ and $V(k)$ in Lemma \ref{lemma: prelim_rates_general_stepsize} and Lemma \ref{lem: rates_general_stepsize} respectively.
		\begin{lemma}
			\label{lem: UVW(K_1) bounds}
			{\sp Suppose Assumptions \ref{asp: gradient samples}-\ref{asp: mu-L_convexity} hold.} Under Algorithm \eqref{eq: x_k}, we have 
			\begin{equation}
				U(k)\le \frac{\hat{X}}{n}, \, V(k)\le \hat{X},\; \forall k\ge 0.
			\end{equation}
		\end{lemma}
		\begin{proof}
			By definitions of $U(k)$, $V(k)$, and Lemma \ref{lem: bounded_iterates_estimate_general_step},  we have
			\begin{align*}
				U(k) &=\bE[\|\ox(k)-x_*\|^2]\\
				&\le \frac{1}{n}\bE[\|\mx(k)-\mathbf{1}x_*^{\T}\|^2]\le \frac{\hat{X}}{n},\\
				V(k) &=\bE[\|\mx(k)-\mathbf{1}\ox(k)^{\T}\|^2]\\
				& \le \bE[\|\mx(k)-\mathbf{1}x_*^{\T}\|^2]\le \hat{X}.
			\end{align*}
			%		Noticing that $K_1= \left\lceil\frac{24 \theta L^2}{(1-\rho_w^2)\mu^2}\right\rceil$,
			%		{\sp we have $\alpha_{K_1-K}=\frac{\theta}{\mu K_1}\le \frac{\mu(1-\rho_w^2)}{24L^2}$, and hence}
			%		\begin{multline*}
			%			W(K_1-K)
			%			= U(K_1-K)+\frac{12\alpha_{K_1-K}L^2}{n\mu(1-\rho_w^2)}V(K_1-K)\\
			%			\le U(K_1-K)+\frac{V(K_1-K)}{2n}.
			%		\end{multline*}
			%		{\sp Since 
			%			\begin{align*}
			%				& U(K_1-K)+\frac{V(K_1-K)}{2n}\\
			%				= & \bE[\|\ox(K_1-K)-x_*\|^2]+\frac{\bE[\|\mx(K_1-K)-\mathbf{1}\ox(K_1-K)\|^2]}{2n}\\
			%				\le & \frac{\bE[\|\mx(K_1-K)-\mathbf{1}x_*\|^2]}{n},
			%			\end{align*}
			%			we conclude from Lemma \ref{lem: bounded_iterates_estimate_general_step} that $W(K_1-K)\le \frac{\hat{W}}{n}$.}
		\end{proof}
	}
	
	Denote {\sp an auxiliary counter}
	\begin{equation}
		\label{tilde k}
		\tilde{k}:=k+K,\quad \forall k\ge 0.
	\end{equation}
	{\sp Our strategy is to first show that the consensus error of Algorithm \eqref{eq: x_k} decays as $V(k)=\mathcal{O}(\frac{1}{k^2})$ based on Lemma \ref{lem: consensus_error_contraction}, since $U(k)$ does not appear explicitly in relation \eqref{V(k+1) bound_pre}. }
	\begin{lemma}
		\label{lemma: prelim_rates_general_stepsize}
		{\sp Suppose Assumptions \ref{asp: gradient samples}-\ref{asp: mu-L_convexity} hold.} {\sp Let 
			\begin{equation}
				\label{K1_definition}
				K_1:= \left\lceil  \max\left\{2K,\frac{16}{1-\rho_w^2}\right\} \right\rceil.
		\end{equation}} 
		Under Algorithm \eqref{eq: x_k} with stepsize \eqref{Stepsize}, for all $k\ge  K_1-K$, we have
		\begin{align*}
			V(k)\le \frac{\hV}{\tilde{k}^2},
		\end{align*}
		{\sp where
			\begin{equation}
				\label{def: hV}
				\hV := \max\left\{K_1^2 \hat{X}, \frac{8\theta^2\rho_w^2 c_1}{\mu^2 (1-\rho_w^2)}\right\},
			\end{equation}
			with
			\begin{align}
				\label{def: c_1}
				c_1 := 2\left(\frac{3}{1-\rho_w^2}+M\right)\left(L^2\hat{X}+ \|\nabla F(\mathbf{1}x_*^{\T})\|^2\right)+n\sigma^2.
		\end{align}}
	\end{lemma}
	\begin{proof}
		{\sp From Lemma \ref{lem: consensus_error_contraction} and  Lemma \ref{lem: bounded_iterates_estimate_general_step}, for $k\ge 0$,
			\begin{align}
				V(k+1)\le \frac{\left(3+\rho_w^2\right)}{4}V(k)+\alpha_k^2\rho_w^2 c_1,   \label{V(k+1) bound}
			\end{align}
			with $c_1$ defined in \eqref{def: c_1}. From the definitions of  $\alpha_k$ and $\tilde{V}(k)$ in \eqref{Stepsize} and \eqref{def: tUtV} respectively, we know that when $k\ge K$,
			\begin{align*}
				\tilde{V}(k+1)
				\le \frac{(3+\rho_w^2)}{4}\tilde{V}(k) +\frac{\theta^2\rho_w^2 c_1}{\mu^2}\left(\frac{1}{k^2}\right).
			\end{align*}
			
			We now prove the lemma by induction. For $k=K_1$, we know $\tilde{V}(k)= \frac{K_1^2\tilde{V}(K_1)}{K_1^2}\le \frac{\hV}{k^2}$ from Lemma \ref{lem: UVW(K_1) bounds}. Now suppose $\tilde{V}(k)\le \frac{\hV}{k^2}$ for some $k\ge K_1$, then
			\begin{align*}
				\tilde{V}(k+1)\le \frac{(3+\rho_w^2)}{4}\frac{\hV}{k^2}+\frac{\theta^2\rho_w^2 c_1}{\mu^2}\left(\frac{1}{k^2}\right).
			\end{align*}
		To show that $\tilde{V}(k+1)\le \frac{\hV}{(k+1)^2}$, it is sufficient to show
		\begin{equation*}
			\frac{(3+\rho_w^2)}{4}\frac{\hV}{k^2}+\frac{\theta^2\rho_w^2 c_1}{\mu^2}\left(\frac{1}{k^2}\right) \le \frac{\hV}{(k+1)^2},
		\end{equation*}
	or equivalently,
	\begin{equation}
		\label{hV_inequality}
		\hV \ge \frac{\theta^2\rho_w^2 c_1}{\mu^2}\left[\left(\frac{k}{k+1}\right)^2-\frac{3+\rho_w^2}{4}\right]^{-1}.
	\end{equation}
Since $k\ge K_1\ge \frac{16}{1-\rho_w^2}$, we have
\begin{multline*}
	\left(\frac{k}{k+1}\right)^2-\frac{3+\rho_w^2}{4}=-\frac{2}{k+1}+\frac{1}{(k+1)^2}+\frac{1-\rho_w^2}{4}\\
	\ge \frac{1-\rho_w^2}{8}.
\end{multline*}
Hence relation \eqref{hV_inequality} is satisfied with $\hV\ge \frac{8\theta^2\rho_w^2 c_1}{\mu^2 (1-\rho_w^2)}$.
			We then have for all $k\ge K_1$,
			\begin{equation*}
				\tilde{V}(k)\le \frac{1}{k^2}\max\left\{K_1^2 \hat{X},\frac{8\theta^2\rho_w^2 c_1}{\mu^2 (1-\rho_w^2)}\right\}.
			\end{equation*}
			Recalling the connection of $\tilde{V}(k)$ and $V(k)$, we conclude that
			$V(k)\le \frac{\hV}{\tilde{k}^2}$ for all $k\ge K_1-K$.
		}
	\end{proof}
	
	{\sp To prove the sublinear convergence of $U(k)$, we start with a useful lemma which provides lower and upper bounds for the product of a decreasing sequence. Such products arise in the convergence proof for $U(k)$ and our main convergence results in Section \ref{sec: main_results}.}
	\begin{lemma}
		\label{lem: product}
		For any $1<a< k$ ($a\in \mathbb{N}$) and $1<\gamma\le a/2$,
		\begin{equation*}
			\frac{a^{2\gamma}}{k^{2\gamma}}\le\prod_{t=a}^{k-1} \left(1-\frac{\gamma}{t}\right)\le \frac{a^{\gamma}}{k^{\gamma}}.
		\end{equation*}
	\end{lemma}
	\begin{proof}
		Denote $G(k):=\prod_{t=a}^{k-1} \left(1-\frac{\gamma}{t}\right)$. We first show that $G(k)\le \frac{a^{\gamma}}{k^{\gamma}}$. Suppose $G(k)\le \frac{M_1}{k^{\gamma}}$ for some $M_1>0$ and $k\ge a$.
		Then,
		\begin{align*}
			G(k+1)= \left(1-\frac{\gamma}{k}\right)G(k)\le  \left(1-\frac{\gamma}{k}\right)\frac{M_1}{k^{\gamma}}\le \frac{M_1}{(k+1)^{\gamma}}.
		\end{align*}
		To see why the last inequality holds, note that $\left(\frac{k}{k+1}\right)^{\gamma}\ge 1-\frac{\gamma}{k}$.
		Taking $M_1=a^{\gamma}$, we have $G(a)=1\le \frac{M_1}{a^{\gamma}}$. The desired relation then holds for all $k> a$.
		
		Now suppose $G(k)\ge \frac{M_2}{k^{2\gamma}}$ for some $M_2>0$ and $k\ge a$. It follows that
		\begin{align*}
			G(k+1)= \left(1-\frac{\gamma}{k}\right)G(k)\ge  \left(1-\frac{\gamma}{k}\right)\frac{M_2}{k^{2\gamma}}\ge \frac{M_2}{(k+1)^{2\gamma}},
		\end{align*}
		where the last inequality follows from
		$\left(\frac{k}{k+1}\right)^{2\gamma}\le 1-\frac{\gamma}{k}$ (noting that $\gamma\le a/2\le k/2$).
		Taking $M_2=a^{2\gamma}$, we have $G(a)=1\le \frac{M_2}{a^{2\gamma}}$. The desired relation then holds for all $k> a$.
		%	See Appendix \ref{proof lem: product}.
	\end{proof}
	
	{\sp In light of Lemma \ref{lem: optimization_error_contraction} and the other supporting lemmas, we establish the $\mathcal{O}(\frac{1}{k})$ convergence rate of $U(k)$ in the following lemma.}
	\begin{lemma}
		\label{lem: rates_general_stepsize}
		{\sp Suppose Assumptions \ref{asp: gradient samples}-\ref{asp: mu-L_convexity} hold.} Under Algorithm \eqref{eq: x_k} with stepsize \eqref{Stepsize}, suppose $\theta>2$.
		%		\footnote{The condition $\theta>2$ can be easily relaxed to the case where $\theta>1$.} 
		We have
		\begin{multline*}
			U(k) \leq \frac{\theta^2{\sp c_2}}{(1.5\theta-1)n\mu^2 \tilde{k}} {\sp+ \frac{K_1^{1.5\theta}}{\tilde{k}^{1.5\theta}} \frac{\hat{X}}{n}}\\
		 \quad +\left[\frac{3\theta^2(1.5\theta-1){\sp c_2}}{(1.5\theta-2)n\mu^2}+\frac{6\theta L^2 \hV}{(1.5\theta-2)n\mu^2}\right]\frac{1}{\tilde{k}^2},
		\end{multline*}
		for all $k\ge K_1-K$, {\sp where
			\begin{equation}
				\label{def: c_2}
				c_2 : =\frac{2ML^2}{n}\hat{X}+\bar{M}.
		\end{equation}}
	\end{lemma}
	\begin{proof}
		In light of Lemma \ref{lem: optimization_error_contraction} {\sp and Lemma \ref{lem: bounded_iterates_estimate_general_step}}, for all $k\ge 0$, we have
		\begin{align*}
			U(k+1) \le \left(1-\frac{3}{2}\alpha_k\mu\right)U(k)+\frac{3\alpha_k L^2}{n\mu}V(k)+\frac{\alpha_k^2{\sp c_2}}{n}.
		\end{align*}
		Recalling the definitions of $\tilde{U}(k)$ and $\tilde{V}(k)$, for all $k\ge K$,
		\begin{equation*}
			\tilde{U}(k+1)\le  \left(1-\frac{3\theta}{2k}\right)\tilde{U}(k)+\frac{3\theta L^2}{n\mu^2}\frac{\tilde{V}(k)}{k}+\frac{\theta^2{\sp c_2}}{n\mu^2}\frac{1}{k^2}.
		\end{equation*}
		Therefore,
		\begin{multline*}
			\tilde{U}(k) \leq \prod_{t=K_1}^{k-1} \left(1-\frac{3\theta}{2t}\right) \tilde{U}(K_1)\\
			+ \sum_{t=K_1}^{k-1}\left(\prod_{j=t+1}^{k-1} \left(1-\frac{3\theta}{2j}\right)\right) \left(\frac{\theta^2{\sp c_2}}{n\mu^2}\frac{1}{t^2} +\frac{3\theta L^2}{n\mu^2}\frac{\tilde{V}(t)}{t}\right).
		\end{multline*}
		From Lemma \ref{lem: product},
		\begin{align*}
			& \tilde{U}(k) 
			\leq \frac{K_1^{1.5\theta}}{k^{1.5\theta}} \tilde{U}(K_1)\\
			&\quad + \sum_{t=K_1}^{k-1}\frac{(t+1)^{1.5\theta}}{k^{1.5\theta}}\left(\frac{\theta^2{\sp c_2}}{n\mu^2t^2} +\frac{3\theta L^2}{n\mu^2}\frac{\tilde{V}(t)}{t}\right)\notag\\
			&  =  \frac{1}{k^{1.5\theta}}\frac{\theta^2{\sp c_2}}{n\mu^2}\sum_{t=K_1}^{k-1}\frac{(t+1)^{1.5\theta}}{t^2}+\frac{K_1^{1.5\theta}}{k^{1.5\theta}} \tilde{U}(K_1)\\
			& \quad + \sum_{t=K_1}^{k-1}\frac{(t+1)^{1.5\theta}}{k^{1.5\theta}} \frac{3\theta L^2}{n\mu^2}\frac{\tilde{V}(t)}{t}.
		\end{align*}
		In light of Lemma \ref{lemma: prelim_rates_general_stepsize}, when $k\ge K_1$,
		$\tilde{V}(k)\le  \frac{\hV}{k^2}$.
		Hence,
		\begin{align*}
			&\tilde{U}(k) -\frac{1}{k^{1.5\theta}}\frac{\theta^2{\sp c_2}}{n\mu^2}\sum_{t=K_1}^{k-1}\frac{(t+1)^{1.5\theta}}{t^2}-\frac{K_1^{1.5\theta}}{k^{1.5\theta}} \tilde{U}(K_1) \\
			& \leq 
			\sum_{t=K_1}^{k-1}\frac{(t+1)^{1.5\theta}}{k^{1.5\theta}} \frac{3\theta L^2}{n\mu^2}\frac{\hV}{t^3}\\
			&  = \frac{1}{k^{1.5\theta}}\frac{3\theta L^2 \hV}{n\mu^2}\sum_{t=K_1}^{k-1}\frac{(t+1)^{1.5\theta}}{t^3}.
		\end{align*}
		However, we have for any $b>a\ge K_1$,
		\begin{align*}
			& \sum_a^b \frac{(t+1)^{1.5\theta}}{t^2}\\
			&\le \sum_a^{b-2} \left[\frac{(t+1)^{1.5\theta}}{(t+1)^2}+3\frac{(t+1)^{1.5\theta}}{(t+1)^3}\right]+\frac{b^{1.5\theta}}{(b-1)^2}\\
			& \quad+\frac{(b+1)^{1.5\theta}}{b^2} \\
			& \le \int_a^b \left(t^{1.5\theta-2}+3t^{1.5\theta-3}\right)dt+\frac{2(b+1)^{1.5\theta}}{b^2}\\
			& \le  \frac{b^{1.5\theta-1}}{1.5\theta-1}+\frac{3b^{1.5\theta-2}}{1.5\theta-2}+3b^{1.5\theta-2},
		\end{align*}
		{\sp where the last inequality comes from the fact that $(\frac{b+1}{b})^{1.5\theta}\le (\frac{4\theta+1}{4\theta})^{1.5\theta}\le \exp(\frac{3}{8})< \frac{3}{2}$ (given that $b> K_1\ge 4\theta$),}
		{\sp and
			\begin{align*}
				& \sum_a^b \frac{(t+1)^{1.5\theta}}{t^3}\le \frac{3}{2}\sum_a^b t^{1.5\theta-3} \le \frac{3}{2}\int_a^{b+1} t^{1.5\theta-3}dt \\
				& \le \frac{3}{2}\frac{(b+1)^{1.5\theta-2}}{1.5\theta-2} \le \frac{2b^{1.5\theta-2}}{1.5\theta-2}.
		\end{align*}}
		Hence, {\sp for all $k\ge K_1$,}
		\begin{multline*}
			\tilde{U}(k) \leq \frac{\theta^2{\sp c_2}}{(1.5\theta-1)n\mu^2 k}+\frac{3\theta^2(1.5\theta-1){\sp c_2}}{(1.5\theta-2)n\mu^2}\frac{1}{k^2}\\
			+\frac{K_1^{1.5\theta}}{k^{1.5\theta}} \tilde{U}(K_1)+\frac{6\theta L^2 \hV}{(1.5\theta-2)n\mu^2}\frac{1}{k^2}.
		\end{multline*}
		Recalling Lemma \ref{lem: UVW(K_1) bounds} and the definition of $\tilde{U}(k)$ yields the desired result.
	\end{proof}
	\begin{remark}{\sp 
			Notice that the convergence rate established in Lemma \ref{lem: rates_general_stepsize} is not asymptotically the same as centralized stochastic gradient descent, since the constant $c_2$ contains information about the initial solutions. In the next section, we will improve the convergence result and show that DSGD indeed performs as well as centralized SGD asymptotically.
		}	
	\end{remark}
	
	\section{Main Results}
	\label{sec: main_results}
	
	In this section, we perform a non-asymptotic analysis of network independence for Algorithm \eqref{eq: x_k}. Specifically, in {\sp Theorem~\ref{Thm: best rate},
		we show that 
		\begin{align*} \frac{1}{n}\sum_{i=1}^n\bE[\|x_i(k)-x_*\|^2]& =\frac{\theta^2 \sp{\bar{M}}}{(2\theta-1)n\mu^2 \tilde{k}} \\ &  +\mathcal{O}\left(\frac{1}{\sqrt{n}(1-\rho_w)}\right)\frac{1}{\tilde{k}^{1.5}}\\ & +\mathcal{O}\left(\frac{1}{(1-\rho_w)^2}\right)\frac{1}{\tilde{k}^2}, 
		\end{align*} where the first term is network independent and the second and third (higher-order) terms} depends on $(1-\rho_w)$. {\sp Then we compare the result} with centralized stochastic gradient descent and show that asymptotically, the two methods have the same convergence rate $\sp{\frac{\theta^2 \sp{\bar{M}}}{(2\theta-1)n\mu^2 \tilde{k}}}$.
	In addition, it takes $K_T=\mathcal{O}\left(\frac{n}{(1-\rho_w)^2}\right)$ time for Algorithm \eqref{eq: x_k} to reach this asymptotic rate of convergence. Finally, we construct a ``hard" optimization problem for which we show the transient time $K_T$ is sharp.
	
	{\sp Our first step is to simplify the presentation of the convergence results in Lemma \ref{lemma: prelim_rates_general_stepsize} and Lemma \ref{lem: rates_general_stepsize}, so that we can utilize them for deriving improved convergence rates conveniently. For this purpose, we first estimate the  constants $\hat{X}$, $\hV$, $c_1$ and $c_2$} appearing in {\sp the two lemmas}  and derive their dependency on the network size $n$, the spectral gap $(1-\rho_w)$, {\sp the summation of initial optimization errors $\sum_{i=1}^{n}\|x_i(0)-x_*\|^2$, and $\sum_{i=1}^{n}\|\nabla f_i(x_*)\|^2$, where the last term can be seen as a measure of the difference among each agent's individual cost functions}.
	\begin{lemma}
		\label{lem: order_constants_general_step}
		{\sp Denote $A : = \sum_{i=1}^{n}\|x_i(0)-x_*\|^2$ and $B := \sum_{i=1}^{n}\|\nabla f_i(x_*)\|^2$.}
		%		 $\|\mx(0)-\mathbf{1}x_*\|^2=\mathcal{O}(n)$, $\|\nabla F(\mathbf{1}x_*^{\T})\|^2=\mathcal{O}(n)$. 
		\sp{Then,
			\begin{align*}
				& \hat{X}=\mathcal{O}(A+B+n),\quad c_1 = \mathcal{O}\left(\frac{A+B+n}{1-\rho_w}\right),\\
				& \hV=\mathcal{O}\left(\frac{A+B+n}{(1-\rho_w)^2}\right), \quad c_2 = \mathcal{O}\left(\frac{A+B+n}{n}\right).
		\end{align*}}
	\end{lemma}
	\begin{proof}
		{\sp We first estimate the constant $\hat{X}$ which appears in the definition (\ref{def: hV}) for $\hV$.
			From Lemma \ref{lem: bounded_iterates_estimate_general_step},
			\begin{multline*}
				\hat{X}\le \|\mx(0)-\mathbf{1}x_*\|^2+\frac{9\|\nabla F(\mathbf{1}x_*^{\T})\|^2}{\mu^2}+\frac{n\sigma^2}{{\sp(1+M)}L^2}\\
				= \mathcal{O}(A+B+n).
			\end{multline*}
%			Since $\|\mx(0)-\mathbf{1}x_*\|^2=\mathcal{O}(n)$ and $\|\nabla F(\mathbf{1}x_*^{\T})\|^2=\mathcal{O}(n)$, we have
%			$
%			\hat{X}=\mathcal{O}(n).
%			$
			From the definition of $c_1$ in (\ref{def: c_1}),
			\begin{multline*}
				c_1 = 2\left(\frac{3}{1-\rho_w^2}+M\right)\left(L^2\hat{X}+ \|\nabla F(\mathbf{1}x_*^{\T})\|^2\right)+n\sigma^2\\
				=\mathcal{O}\left(\frac{A+B+n}{1-\rho_w}\right).
			\end{multline*}
%			Noting that $\|\nabla F(\mathbf{1}x_*^{\T})\|^2=\mathcal{O}(n)$ and $\hat{X}=\mathcal{O}(n)$, we have
%			$
%			c_1=\mathcal{O}\left(\frac{n}{1-\rho_w}\right). 
%			$
			Noting that $K_1=\mathcal{O}(\frac{1}{1-\rho_w})$, by definition,
			\begin{align*}
				\hV = \max\left\{K_1^2 \hat{X}, \frac{8\theta^2\rho_w^2 c_1}{\mu^2 (1-\rho_w^2)}\right\}= \mathcal{O}\left(\frac{A+B+n}{(1-\rho_w)^2}\right).
			\end{align*}
			From the definition of $c_2$ in \eqref{def: c_2}, we have
			\begin{equation*}
				c_2 = \frac{2ML^2}{n}\hat{X}+\bar{M} =\mathcal{O}\left(\frac{A+B+n}{n}\right).
		\end{equation*}}
		%		See Appendix \ref{proof lem: order_constants_general_step}.
	\end{proof}
	
	{\sp In light of Lemma \ref{lem: order_constants_general_step}, the convergence result of $V(k)$ given in Lemma \ref{lemma: prelim_rates_general_stepsize} can be easily simplified since $\hV$ is the only constant.}   {Regarding the optimization error $U(k)$,} in light of Lemma \ref{lem: bounded_iterates_estimate_general_step}, Lemma \ref{lemma: prelim_rates_general_stepsize}, Lemma \ref{lem: rates_general_stepsize} and Lemma \ref{lem: order_constants_general_step}, we have the following corollary {\sp which simplifies the presentation of the convergence result in Lemma \ref{lem: rates_general_stepsize}.
		%		and also bound  $\frac{1}{n}\sum_{i=1}^n\bE[\|x_i(k)-x_*\|^2]$ that directly measure the performance of DSGD.
	}
	\begin{corollary} 
		\label{cor: U(k)_general_step}
		{\sp Suppose Assumptions \ref{asp: gradient samples}-\ref{asp: mu-L_convexity} hold.} 
%		Assume in addition that $\sum_{i=1}^{n}\|x_i(0)-x_*\|^2=\mathcal{O}(n)$ and $\sum_{i=1}^{n}\|\nabla f_i(x_*)\|^2=\mathcal{O}(n)$. 
		Under Algorithm \eqref{eq: x_k} with {\sp stepsize \eqref{Stepsize} and assuming }$\theta>2$, when $k\ge K_1-K$,
		\begin{align*}
			U(k) \le  \frac{\theta^2{\sp c_2}}{(1.5\theta-1)n\mu^2\tilde{k}}+ \frac{c}{\tilde{k}^2},
		\end{align*}
		where
		\begin{align*}
			c=\mathcal{O}\left(\sp{\frac{A+B+n}{n(1-\rho_w)^2}}\right).
		\end{align*}
		%		In addition,
		%		\begin{multline*}
		%		\frac{1}{n}\sum_{i=1}^n\bE[\|x_i(k)-x_*\|^2]
		%		\le \frac{\theta^2\sp{c_2}}{(1.5\theta-1)n\mu^2}\frac{1}{\tilde{k}}\\
		%		+\mathcal{O}\left(\frac{1}{(1-\rho_w)^2}\right)\frac{1}{\tilde{k}^2}.
		%		\end{multline*}
	\end{corollary}
	\begin{proof}
		From Theorem \ref{lem: rates_general_stepsize} and Lemma \ref{lem: order_constants_general_step}, when $k\ge K_1-K=\mathcal{O}(\frac{1}{1-\rho_w})$,
		\begin{align}
			\label{U(k) bound proof}
			U(k) \le &  \frac{\theta^2 {\sp c_2}}{(1.5\theta-1)n\mu^2 \tilde{k}}{\sp+ \frac{K_1^{1.5\theta-2}}{\tilde{k}^{1.5\theta-2}} \frac{\hat{X}}{n}\frac{K_1^2}{\tilde{k}^2}}\notag\\
			& +\left[\frac{3\theta^2(1.5\theta-1) {\sp c_2}}{(1.5\theta-2)n\mu^2}+\frac{6\theta L^2 \hV}{(1.5\theta-2)n\mu^2}\right]\frac{1}{\tilde{k}^2}\notag\\
			= &  \frac{\theta^2 {\sp c_2}}{(1.5\theta-1)n\mu^2}\frac{1}{\tilde{k}}+\mathcal{O}\left(\sp{\frac{A+B+n}{n(1-\rho_w)^2}}\right)\frac{1}{\tilde{k}^2}.
		\end{align}	
		%	From Lemma \ref{lemma: prelim_rates_general_stepsize}, when $k\ge K_1$,
		%	\begin{align*}
		%	V(k)\le  & p_0^{\tilde{k}-K_1}\hat{X}+\frac{V_1}{\tilde{k}^2}+\frac{V_2}{\tilde{k}^3} = p_0^{\tilde{k}-K_1}\mathcal{O}(n)+\mathcal{O}\left(\frac{n}{(1-\rho_w)^2}\right)\frac{1}{\tilde{k}^2}\\
		%	= & \mathcal{O}\left(\frac{n}{(1-\rho_w)^2}\right)\frac{1}{\tilde{k}^2}.
		%	\end{align*}
		%		Regarding $\frac{1}{n}\sum_{i=1}^n\bE[\|x_i(k)-x_*\|^2]$, we have
		%		\begin{align*}
		%			& \frac{1}{n}\sum_{i=1}^n\bE[\|x_i(k)-x_*\|^2]\\
		%			&=\bE[\|\bar{x}(k)-x_*\|^2]+\frac{1}{n}\sum_{i=1}^n\bE[\|x_i(k)-\bar{x}\|^2]\\
		%			& =U(k)+\frac{V(k)}{n}\\
		%			&\le \frac{\theta^2{\sp c_2}}{(1.5\theta-1)n\mu^2}\frac{1}{\tilde{k}}+\mathcal{O}\left(\frac{1}{(1-\rho_w)^2}\right)\frac{1}{\tilde{k}^2},
		%		\end{align*}
		%		where we invoked {\sp relation \eqref{U(k) bound proof} and Lemma \ref{lem: order_constants_general_step}} for the last inequality.
	\end{proof}
	
	{\sp Let $\frac{1}{n}\sum_{i=1}^n\bE[\|x_i(k)-x_*\|^2]$, the average optimization error for each agent to measure the performance of DSGD.} In the following theorem, we improve the result of Corollary~\ref{cor: U(k)_general_step} with further analysis {\sp and derive the main convergence result for Algorithm \eqref{eq: x_k}}. 
	\begin{theorem}
		\label{Thm: best rate}
		{\sp Suppose Assumptions \ref{asp: gradient samples}-\ref{asp: mu-L_convexity} hold.} 
%		Assume in addition that $\sum_{i=1}^{n}\|x_i(0)-x_*\|^2=\mathcal{O}(n)$ and $\sum_{i=1}^{n}\|\nabla f_i(x_*)\|^2=\mathcal{O}(n)$. 
		Under Algorithm \eqref{eq: x_k} with {\sp stepsize \eqref{Stepsize} and assuming} $\theta>2$, when $k\ge K_1-K$,
		\begin{multline}
			\label{Best rate}
			\frac{1}{n}\sum_{i=1}^n\bE[\|x_i(k)-x_*\|^2] \le \frac{\theta^2 \sp{\bar{M}}}{(2\theta-1)n\mu^2 \tilde{k}}\\ +\mathcal{O}\left(\sp{\frac{\sqrt{A+B+n}}{n(1-\rho_w)}}\right)\frac{1}{\tilde{k}^{1.5}}+\mathcal{O}\left(\sp{\frac{A+B+n}{n(1-\rho_w)^2}}\right)\frac{1}{\tilde{k}^2}.
		\end{multline}
		%	\begin{equation*}
		%	U(k)\le \frac{\theta^2\sigma^2}{(2\theta-1)n\mu^2 k} +\mathcal{O}\left(\frac{1}{\sqrt{n}(1-\rho_w)}\right)\frac{1}{k^{1.5}}+\mathcal{O}\left(\frac{1}{(1-\rho_w)^{1.75}}\right)\frac{1}{k^{1.75}}.
		%	\end{equation*}
	\end{theorem}
	\begin{proof}
		For $k\ge K_1-K$, in light of Lemma \ref{lem: contraction_mu-L_convexity} and Lemma \ref{lem: optimization_error_contraction_pre},
		\begin{equation*}
			\begin{split}
				& U(k+1)\\
				& \le (1-\alpha_k\mu)^2U(k)
				+\frac{2\alpha_k L}{\sqrt{n}}\bE[\|\ox(k)-x_*\|\|\mx(k)-\mathbf{1}\ox(k)^{\T}\|]\\
				& \quad +\frac{\alpha_k^2L^2}{n}V(k) +{\sp \alpha_k^2\left(\frac{2ML^2}{n^2}\bE[\|\mx(k)-\mathbf{1}x_*^{\T}\|^2]+\frac{\bar{M}}{n}\right)}\\
				& \le (1-\alpha_k\mu)^2U(k)
				+\frac{2\alpha_k L}{\sqrt{n}}\sqrt{U(k)V(k)}+\frac{\alpha_k^2L^2}{n}V(k)\\
				& \quad +{\sp \alpha_k^2\left[\frac{2ML^2}{n^2}(nU(k)+V(k))+\frac{\bar{M}}{n}\right]}\\
				& {\sp = (1-2\alpha_k\mu)U(k) + \alpha_k^2\left(\mu^2+\frac{2ML^2}{n}\right)U(k)}\\
				& \quad {\sp +\frac{2\alpha_k L}{\sqrt{n}}\sqrt{U(k)V(k)} + \frac{\alpha_k^2 L^2}{n}\left(1+\frac{2M}{n}\right)V(k)+\frac{\alpha_k^2\bar{M}}{n},}
			\end{split}
		\end{equation*}
		where the second inequality follows from the Cauchy-Schwarz inequality {\sp (see \cite{williams1991probability}, page 62) and the fact that $\|\mx(k)-\mathbf{1}x_*^{\T}\|^2=\|\mx(k)-\mathbf{1}\bar{x}(k)\|^2+n\|\bar{x}(k)-x_*\|^2$}.
		
		Recalling the definitions of $\tilde{U}(k)$ and $\tilde{V}(k)$, {\sp when $k\ge K_1$,
			\begin{multline*}
				\tilde{U}(k+1)\le  \left(1-\frac{2\theta}{k}\right)\tilde{U}(k)+\frac{\theta^2 }{k^2}\left(1+\frac{2ML^2}{n\mu^2}\right)\tilde{U}(k)\\
				+\frac{2\theta L}{\sqrt{n}\mu}\frac{\sqrt{\tilde{U}(k)\tilde{V}(k)}}{k}
				+\frac{\theta^2 L^2}{n\mu^2}\left(1+\frac{2M}{n}\right)\frac{\tilde{V}(k)}{k^2}+\frac{\theta^2\bar{M}}{n\mu^2}\frac{1}{k^2}.
		\end{multline*}}
		Therefore, {\sp by denoting $c_3 := 1+\frac{2ML^2}{n\mu^2}$ and $c_4 := 1+\frac{2M}{n}$, we have} \begin{small}
		\begin{align*}
			& \tilde{U}(k) \leq  \left(\prod_{t=K_1}^{k-1} \left(1-\frac{2\theta}{t}\right)\right) \tilde{U}(K_1) \\
			& \quad+ \sum_{t=K_1}^{k-1}\left(\prod_{i=t+1}^{k-1} \left(1-\frac{2\theta}{i}\right)\right) \\
			& \quad \cdot\left(\frac{\theta^2 \sp{\bar{M}}}{n\mu^2t^2} + \frac{\theta^2 {\sp c_3}\tilde{U}(t)}{t^2}+\frac{2\theta L}{\sqrt{n}\mu}\frac{\sqrt{\tilde{U}(t)\tilde{V}(t)}}{t}+\frac{\theta^2L^2 {\sp c_4}}{n\mu^2}\frac{\tilde{V}(t)}{t^2}\right)
		\end{align*} \end{small}
		From Lemma \ref{lem: product},
		\begin{equation*}
			\begin{split}
				& \tilde{U}(k) 
				\leq \frac{K_1^{2\theta}}{k^{2\theta}} \tilde{U}(K_1) + \sum_{t=K_1}^{k-1}\frac{(t+1)^{2\theta}}{k^{2\theta}}\left(\frac{\theta^2\sp{\bar{M}}}{n\mu^2t^2} + \frac{\theta^2 {\sp c_3}\tilde{U}(t)}{t^2}\right.\\
				& \quad \left.+\frac{2\theta L}{\sqrt{n}\mu}\frac{\sqrt{\tilde{U}(t)\tilde{V}(t)}}{t}+\frac{\theta^2L^2{\sp c_4}}{n\mu^2}\frac{\tilde{V}(t)}{t^2}\right)\\
				& = \frac{1}{k^{2\theta}}\frac{\theta^2\sp{\bar{M}}}{n\mu^2}\sum_{t=K_1}^{k-1}\frac{(t+1)^{2\theta}}{t^2}+\frac{K_1^{2\theta}}{k^{2\theta}} \tilde{U}(K_1)\\
				& \quad + \sum_{t=K_1}^{k-1}\frac{(t+1)^{2\theta}}{k^{2\theta}} \left(\frac{\theta^2{\sp c_3}\tilde{U}(t)}{t^2}+\frac{2\theta L}{\sqrt{n}\mu}\frac{\sqrt{\tilde{U}(t)\tilde{V}(t)}}{t}\right.\\
				& \quad \left.+\frac{\theta^2 L^2{\sp c_4}}{n\mu^2}\frac{\tilde{V}(t)}{t^2}\right).
			\end{split}
		\end{equation*}
		Hence,  by Corollary \ref{cor: U(k)_general_step},
		\begin{equation*}
			\begin{split}
				& \tilde{U}(k) -  \frac{1}{k^{2\theta}}\frac{\theta^2\sp{\bar{M}}}{n\mu^2}\sum_{t=K_1}^{k-1}\frac{(t+1)^{2\theta}}{t^2}-\frac{K_1^{2\theta}}{k^{2\theta}} \tilde{U}(K_1)\\
				& \le \frac{\theta^2 \sp{c_3}}{k^{2\theta}}\sum_{t=K_1}^{k-1}\frac{(t+1)^{2\theta}}{t^2} \left[\frac{\theta^2 {\sp c_2}}{(1.5\theta-1)n\mu^2 t}+ \frac{c}{t^2}\right]\\
				& \quad+\frac{1}{k^{2\theta}}\frac{2\theta L}{\sqrt{n}\mu}\sum_{t=K_1}^{k-1}\frac{(t+1)^{2\theta}}{t}\sqrt{\frac{\theta^2{\sp c_2}}{(1.5\theta-1)n\mu^2}\frac{1}{t}+ \frac{c}{t^2}}\sqrt{\frac{\hV}{t^2}}\\
				& \quad+\frac{1}{k^{2\theta}}\frac{\theta^2 L^2{\sp c_4}}{n\mu^2}\sum_{t=K_1}^{k-1}\frac{(t+1)^{2\theta}}{t^2}\frac{\hV}{t^2}.
			\end{split}
		\end{equation*}
		{\sp Since 
			\begin{equation*}
				\sqrt{\frac{\theta^2{\sp c_2}}{(1.5\theta-1)n\mu^2}\frac{1}{t}+ \frac{c}{t^2}}\sqrt{\frac{\hV}{t^2}}\le \sqrt{\frac{\theta^2{\sp c_2} \hV}{(1.5\theta-1)n\mu^2}\frac{1}{ t^3}} + \frac{\sqrt{c\hV}}{t^2},
			\end{equation*}
			we have}
		\begin{equation*}
			\begin{split}
				& \tilde{U}(k) -  \frac{1}{k^{2\theta}}\frac{\theta^2\sp{\bar{M}}}{n\mu^2}\sum_{t=K_1}^{k-1}\frac{(t+1)^{2\theta}}{t^2}-\frac{K_1^{2\theta}}{k^{2\theta}} \tilde{U}(K_1)\\
				& \le \frac{\theta^2{\sp c_3}}{k^{2\theta}}\sum_{t=K_1}^{k-1}\frac{(t+1)^{2\theta}}{t^2} \left[\frac{\theta^2{\sp c_2}}{(1.5\theta-1)n\mu^2}\frac{1}{t}+ \frac{c}{t^2}\right]\\
				& \quad +\frac{1}{k^{2\theta}}\frac{2\theta L}{\sqrt{n}\mu}\sum_{t=K_1}^{k-1}\frac{(t+1)^{2\theta}}{t}\left(\sqrt{\frac{\theta^2{\sp c_2} \hV}{(1.5\theta-1)n\mu^2}\frac{1}{ t^3}} + \frac{\sqrt{c\hV}}{t^2}\right)\\
				& \quad+\frac{1}{k^{2\theta}}\frac{\theta^2 L^2{\sp c_4} \hV}{n\mu^2}\sum_{t=K_1}^{k-1}\frac{(t+1)^{2\theta}}{t^4}\\
				& =\frac{1}{k^{2\theta}}\left(\frac{2\theta^2 L\sigma\sqrt{{\sp c_2}\hV}}{\sqrt{1.5\theta-1}n\mu^2}\right)\sum_{t=K_1}^{k-1}\frac{(t+1)^{2\theta}}{t^{2.5}}\\
				& \quad+\frac{1}{k^{2\theta}}\left(\frac{\theta^4{\sp c_3 c_2}}{(1.5\theta-1)n\mu^2}+\frac{2\theta L\sqrt{c \hV}}{\sqrt{n}\mu}\right)\sum_{t=K_1}^{k-1}\frac{(t+1)^{2\theta}}{t^3}\\
				& \quad+ \frac{1}{k^{2\theta}}\left(\theta^2 {\sp c_3}c+\frac{\theta^2L^2 {\sp c_4}\hV}{n\mu^2}\right)\sum_{t=K_1}^{k-1}\frac{(t+1)^{2\theta}}{t^4}.
			\end{split}
		\end{equation*}
		{\sp Notice that $c_2=\mathcal{O}(\frac{A+B+n}{n})$ and $c_3,c_4=\mathcal{O}(1)$.}
		Following a discussion similar to those in the proofs for Theorem \ref{lem: rates_general_stepsize} and Corollary \ref{cor: U(k)_general_step}, we have
		\begin{equation*}
			\begin{split}
				\tilde{U}(k) \le & \frac{\theta^2\sp{\bar{M}}}{(2\theta-1)n\mu^2 k} +\mathcal{O}\left(\sp{\frac{\sqrt{A+B+n}}{n(1-\rho_w)}}\right)\frac{1}{k^{1.5}}\\
				& +\mathcal{O}\left(\sp{\frac{A+B+n}{n(1-\rho_w)^2}}\right)\frac{1}{k^2}+ \mathcal{O}\left(\sp{\frac{A+B+n}{n(1-\rho_w)^2}}\right)\frac{1}{k^3}\\
				& + \mathcal{O}\left(\sp{\frac{A+B+n}{n(1-\rho_w)^{2\theta}}}\right)\frac{1}{k^{2\theta}}\\
				= 	& \frac{\theta^2\sp{\bar{M}}}{(2\theta-1)n\mu^2 k} +\mathcal{O}\left(\sp{\frac{\sqrt{A+B+n}}{n(1-\rho_w)}}\right)\frac{1}{k^{1.5}}\\
				& +\mathcal{O}\left(\sp{\frac{A+B+n}{n(1-\rho_w)^2}}\right)\frac{1}{k^2}.
			\end{split}
		\end{equation*}
		{\sp Noting that
			\begin{align*}
				& \frac{1}{n}\sum_{i=1}^n\bE[\|x_i(k)-x_*\|^2]\\
				&=\bE[\|\bar{x}(k)-x_*\|^2]+\frac{1}{n}\sum_{i=1}^n\bE[\|x_i(k)-\bar{x}\|^2]\\
				& =U(k)+\frac{V(k)}{n},
				%					&\le \frac{\theta^2{\sp c_2}}{(1.5\theta-1)n\mu^2}\frac{1}{\tilde{k}}+\mathcal{O}\left(\frac{1}{(1-\rho_w)^2}\right)\frac{1}{\tilde{k}^2},
			\end{align*}
			%		\begin{equation*}
			%			\frac{1}{n}\sum_{i=1}^n\bE[\|x_i(k)-x_*\|^2]=U(k)+\frac{V(k)}{n},
			%		\end{equation*}
			and $U(k)=\tilde{U}(k+K)$, in light of the bound on $V(k)$ in Lemma \ref{lemma: prelim_rates_general_stepsize} and the estimate of $\hV$ in Lemma \ref{lem: order_constants_general_step}}, we obtain the desired result.
	\end{proof}
	
	\subsection{Comparison with Centralized Implementation}
	
	We compare the performance of DSGD and centralized stochastic gradient descent (SGD) stated below:
	\begin{equation}
		\label{eq: centralized}
		x(k+1)=x(k)-\alpha_k \tilde{g}(k),
	\end{equation}
	where $\alpha_k:=\frac{\theta}{\mu k}$ ($\theta>1$) and $\tilde{g}(k):=\frac{1}{n}\sum_{i=1}^n g(x(k),\xi_i(k))$. 
	
	First, we derive the convergence rate for SGD which matches the optimal rate for such stochastic gradient methods (see \cite{nemirovski2009robust,rakhlin2012making}). Our result relies on an analysis different from the literature that considered a compact feasible set and uniformly bounded stochastic gradients in expectation.
	
	\begin{theorem}
		\label{Thm: centralized}
		Under the centralized stochastic gradient descent of (\ref{eq: centralized}), suppose $k\ge K_2:=\left\lceil\frac{\theta L}{\mu}\right\rceil$. We have
		\begin{align*}
			\bE[\|x(k)-x_*\|^2]  
			\leq  \frac{\theta^2\sp{\bar{M}}}{(2\theta-1)n\mu^2 k}+\mathcal{O}\left(\frac{1}{n}\right)\frac{1}{k^2}.
		\end{align*}
	\end{theorem}
	\begin{proof}
		Noting that $\alpha_k\le 1/L$ when $k\ge K_2$, {\sp we have from Lemma \ref{lem: oy_k-h_k} that}
		\begin{equation}
			\label{centralized_first_inequality}
			\begin{split}
				& \bE[\|x(k+1)-x_*\|^2\mid \mF(k)]\\
				& =  \bE[\|x(k)-\alpha_k \tilde{g}(k)-x_*\|^2\mid \mF(k)]\\
				& = \|x(k)-\alpha_k \nabla f(x(k))-x_*\|^2+\alpha_k^2\bE[\|\nabla f(x(k))-\tilde{g}(k)\|^2]\\
				& \le  (1-\alpha_k\mu)^2\|x(k)-x_*\|^2\\
				& \quad +\alpha_k^2 \sp{\left(\frac{2ML^2}{n}\|x(k)-x_*\|^2+\frac{\bar{M}}{n}\right)}\\
				& = \left(1-\frac{2\theta}{k}\right)\|x(k)-x_*\|^2\\
				&\quad +\theta^2\left(1{\sp+\frac{2M L^2}{n\mu^2}}\right)\|x(k)-x_*\|^2\frac{1}{k^2}+\frac{\theta^2\sp{\bar{M}}}{n\mu^2}\frac{1}{k^2}.
			\end{split}
		\end{equation} 
		It can be shown first that $\bE[\|x(k)-x_*\|^2]\le \frac{c_5}{k}$ for $k\ge K_2$, where $c_5=\mathcal{O}(\frac{1}{n})$.\footnote{The argument here is similar to that in the proof for {\sp Lemma \ref{lem: rates_general_stepsize}}.} {\sp Denote $\bar{c}_5 := (1+\frac{2M L^2}{n\mu^2})c_5$.}  Then from relation (\ref{centralized_first_inequality}), when $k\ge K_2$,
		\begin{equation*}
			\begin{split}
				\bE[\|x(k)-x_*\|^2] \leq \left(\prod_{t=K_2}^{k-1} \left(1-\frac{2\theta}{t}\right)\right) \bE[\|x(K_2)-x_*\|^2] \\
				+ \sum_{t=K_2}^{k-1}\left(\prod_{i=t+1}^{k-1} \left(1-\frac{2\theta}{i}\right)\right) \left(\frac{\theta^2\sp{\bar{M}}}{n\mu^2t^2} + \frac{\theta^2 \sp{\bar{c}_5}}{t^3}\right).
			\end{split}
		\end{equation*}
		From Lemma \ref{lem: product},
		\begin{equation*}
			\begin{split}
				& \bE[\|x(k)-x_*\|^2]  
				\leq  \frac{K_2^{2\theta}}{k^{2\theta}} \bE[\|x(K_2)-x_*\|^2] \\
				& \quad+ \sum_{t=K_2}^{k-1}\frac{(t+1)^{2\theta}}{k^{2\theta}}\left(\frac{\theta^2\sp{\bar{M}}}{n\mu^2t^2} + \frac{\theta^2 \sp{\bar{c}_5}}{t^3}\right)\\
				& = \frac{1}{k^{2\theta}} \frac{\theta^2\sp{\bar{M}}}{n\mu^2} \sum_{t=K_2}^{k-1}\frac{(t+1)^{2\theta}}{t^2}+\frac{K_2^{2\theta}}{k^{2\theta}} \bE[\|x(K_2)-x_*\|^2] \\
				& \quad+\frac{\theta^2 \sp{\bar{c}_5}}{k^{2\theta}} \sum_{t=K_2}^{k-1}\frac{(t+1)^{2\theta}}{t^3}\\
				& = \frac{\theta^2\sp{\bar{M}}}{(2\theta-1)n\mu^2 k}+\mathcal{O}\left(\frac{1}{n}\right)\frac{1}{k^2}.
			\end{split}
		\end{equation*}
	\end{proof}
	
	Comparing the results of Theorem \ref{Thm: best rate} and Theorem \ref{Thm: centralized}, we can see that asymptotically, DSGD and SGD have the same convergence rate $\frac{\theta^2\sp{\bar{M}}}{(2\theta-1)n\mu^2 k}$. The next corollary identifies the time needed for DSGD to achieve this rate.
	\begin{corollary}[Transient Time]
		\label{cor: transient_time}
		{\sp Suppose Assumptions \ref{asp: gradient samples}-\ref{asp: mu-L_convexity} hold.} 
			Assume in addition that $\sum_{i=1}^{n}\|x_i(0)-x_*\|^2=\mathcal{O}(n)$ and $\sum_{i=1}^{n}\|\nabla f_i(x_*)\|^2=\mathcal{O}(n)$.
		It takes $K_T=\mathcal{O}\left(\frac{n}{(1-\rho_w)^2}\right)$ time for Algorithm \eqref{eq: x_k} to reach the asymptotic rate of convergence, i.e., when $k\ge K_T$, we have  $\frac{1}{n}\sum_{i=1}^n \bE[\|x_i(k) \linebreak[3] - x_*\|^2]\le \frac{\theta^2\sp{\bar{M}}}{(2\theta-1)n\mu^2 k}\mathcal{O}(1)$.
	\end{corollary}
	\begin{proof}
		From (\ref{Best rate}),
		\begin{multline*}
			\frac{1}{n}\sum_{i=1}^n\bE[\|x_i(k)-x_*\|^2] \le \frac{\theta^2\sp{\bar{M}}}{(2\theta-1)n\mu^2 k}\\
			\cdot\left[1+\mathcal{O}\left(\frac{\sqrt{n}}{(1-\rho_w)}\right)\frac{1}{k^{0.5}} +\mathcal{O}\left(\frac{n}{(1-\rho_w)^2}\right)\frac{1}{k}\right].
		\end{multline*}
		Let $K_T$ be such that
		\begin{equation*}
			\mathcal{O}\left(\frac{\sqrt{n}}{(1-\rho_w)}\right)\frac{1}{K_T^{0.5}} +\mathcal{O}\left(\frac{n}{(1-\rho_w)^2}\right)\frac{1}{K_T}=\mathcal{O}(1).
		\end{equation*}
		We then obtain that 
		\begin{equation*}
			K_T=\mathcal{O}\left(\frac{n}{(1-\rho_w)^2}\right).
		\end{equation*}
	\end{proof}
	\begin{remark}
		{\sp By assuming the additional conditions $\sum_{i=1}^{n}\|x_i(0)-x_*\|^2=\mathcal{O}(n)$ and $\sum_{i=1}^{n}\|\nabla f_i(x_*)\|^2=\mathcal{O}(n)$,  motivated by the observation that each of these expression is the sum of $n$ terms, we obtain a cleaner expression of the transient time. In general, we would obtain $K_T=\mathcal{O}\left(\frac{A+B+n}{(1-\rho_w)^2}\right)$.}
	\end{remark}
	\begin{remark}
		{\sp For general connected networks such as line graphs}, if we adopt the Lazy Metropolis rule for choosing the weights $[w_{ij}]$ (see \cite{olshevsky2017linear}), then $\frac{1}{1-\rho_w}=\mathcal{O}(n^2)$, and hence $K_T=\mathcal{O}(n^5)$. {\sp The transient time can be improved for networks with special structures. For example, $\frac{1}{1-\rho_w}$ is constant with high probability for a random Erd\H{o}s-R\'{e}nyi random graph, and consequently $K_T=\mathcal{O}(n)$ on such a graph.}
	\end{remark}

	The next theorem states that the transient time for DSGD to reach the asymptotic convergence rate is lower bounded by $\Omega\left(\frac{n}{(1-\rho_w)^2}\right)$, {\sp that is, under Assumptions \ref{asp: gradient samples}-\ref{asp: mu-L_convexity} and assuming $\sum_{i=1}^{n}\|x_i(0)-x_*\|^2=\mathcal{O}(n)$ and $\sum_{i=1}^{n}\|\nabla f_i(x_*)\|^2=\mathcal{O}(n)$, there exists an optimization problem whose transient time under DSGD is lower bounded by $\Omega\left(\frac{n}{(1-\rho_w)^2}\right)$. This implies the result in Corollary \ref{cor: transient_time} is sharp and can not be improved in general.}
	\begin{theorem}
		\label{Thm: sharpness}
		Suppose Assumptions \ref{asp: gradient samples}-\ref{asp: mu-L_convexity} hold. {\sp Assume in addition that $\sum_{i=1}^{n}\|x_i(0)-x_*\|^2=\mathcal{O}(n)$ and $\sum_{i=1}^{n}\|\nabla f_i(x_*)\|^2=\mathcal{O}(n)$. Then there exists a $\rho_0 \in (0,1)$ such that if $\rho_w \geq \rho_0$, then} the time needed for DSGD to reach the asymptotic convergence rate is lower bounded by $\Omega\left(\frac{n}{(1-\rho_w)^2}\right)$.
	\end{theorem}
	\begin{proof}
		We construct a ``hard'' optimization to prove the claimed result, inspired by \cite{duchi2012dual}. Consider quadratic objective functions $f_i(x):=\frac{1}{2}\|x-x_i^*\|^2$, where $x,x_i^*\in\mathbb{R}$. The optimal solution to Problem (\ref{opt Problem_def}) is given by $x_*=\frac{1}{n}\sum_{i=1}^n x_i^*$.
		%	Assume for simplicity that gradient noise is $0$. 
		The DSGD algorithm implements:
		\begin{equation}
			\label{DSGD_quardratic}
			\mx(k+1) = \mW\left(\mx(k)-\alpha_k (\mx(k)-\mx_*)+\alpha_k\mathbf{n}(k)\right),
		\end{equation}
		where $\mx_*:=[x_1^*,x_2^*,\ldots,x_n^*]^{\T}$, and $\mathbf{n}(k)$ denotes the vector of gradient noise terms. From (\ref{Stepsize}), we use stepsize $\alpha_k=\frac{\theta}{k+K}$ ($\theta>2$), where $K=\lceil2\theta\rceil$  since $\mu=L=1$.
		We rewrite (\ref{DSGD_quardratic}) as
		\begin{equation*}
			\mx(k+1) = (1-\alpha_k)\mW\mx(k) + \alpha_k\mW\mx_*+\alpha_k\mW\mathbf{n}(k).
		\end{equation*}
		It follows that
		\begin{multline*}
			\mx(k+1)-\mathbf{1}\ox(k+1)^{\T} = (1-\alpha_k)\mW(\mx(k)-\mathbf{1}\ox(k)^{\T}) \\
			+ \alpha_k\mW(\mx_*-\mathbf{1}x_*)+\alpha_k\mW(\mathbf{n}(k)-\mathbf{1}\on(k)).
		\end{multline*}
		By induction, we have for all $k>0$,
		\begin{align}
			\label{DSGD_quardratic_induction}
			\begin{split}
				\mx(k)-\mathbf{1}\ox(k)^{\T}= & \left(\prod_{t=0}^{k-1}(1-\alpha_t)\right)\mW^k(\mx(0)-\mathbf{1}\ox(0))\\
				& +\sum_{t=0}^{k-1}\left[\left(\prod_{j=t+1}^{k-1}(1-\alpha_j)\right)\alpha_t\mW^{k-t}\right]\\
				& \cdot\left[(\mx_*-\mathbf{1}x_*)+(\mathbf{n}(t)-\mathbf{1}\on(t))\right].
			\end{split}
		\end{align}
		
		Assume that: (i) the matrix $\mW$ is symmetric; (ii) $\mW \mx_* = \rho_w\mx_*$, i.e., $\mx_*$ is an eigenvector of $\mW$ w.r.t. eigenvalue $\rho_w$ (hence $x_*=\frac{1}{n}\mathbf{1}^{\T}\mx_*=0$); (iii) $\|\nabla F(\mathbf{1}x_*^{\T})\|^2=\|\mathbf{1}x_*-\mx_*\|^2=\|\mx_*\|^2=\Omega(n)$; 
		%	\footnote{\sp For instance, $x_1=x_2=\cdots=x_{\lfloor n/2\rfloor}=-1$, $x_{\lfloor n/2\rfloor+1}=\cdots=x_n=1$ for even $n$, and $x_1=x_2=\cdots=x_{\lfloor n/2\rfloor}=-1$, $x_{\lceil n/2\rceil}=0$, $x_{\lceil n/2\rceil+1}=\cdots=x_n=1$ for odd $n$.} 
		(iv) $\mx(0)=\mx_*$.\footnote{Assumptions (iii) and (iv) correspond to the conditions $\|\mx(0)-\mathbf{1}x_*^{\T}\|^2=\mathcal{O}(n)$ and $\|\nabla F(\mathbf{1}x_*^{\T})\|^2=\mathcal{O}(n)$ assumed in the main results such as Theorem \ref{Thm: best rate} and Corollary \ref{cor: transient_time}.} Then $\ox(0)=x_*=0$, and from relation (\ref{DSGD_quardratic_induction}) it follows
		\begin{multline}
			\mx(k)-\mathbf{1}\ox(k)^{\T}=\left(\prod_{t=0}^{k-1}(1-\alpha_t)\right)\rho_w^k\mx_*\\
			+\sum_{t=0}^{k-1}\left[\left(\prod_{j=t+1}^{k-1}(1-\alpha_j)\right)\alpha_t\rho_w^{k-t}\right]\mx_*+\epsilon(k),
		\end{multline}
		where $\epsilon(k)$ captures the random perturbation caused by gradient noise that has mean zero.
		Therefore, 
		\begin{multline*}
			\bE[\|\mx(k)-\mathbf{1}\ox(k)^{\T}\|^2]\ge \\ \left\|\sum_{t=0}^{k-1}\left[\left(\prod_{j=t+1}^{k-1}(1-\alpha_j)\right)\alpha_t\rho_w^{k-t}\right]\mx_*\right\|^2.
		\end{multline*}
		Recalling the definition $V(k)=\bE[\|\mx(k)-\mathbf{1}\ox(k)^{\T}\|^2$ and $\tilde{V}(k)=V(k-K)$, and noticing that $\alpha_k=\frac{\theta}{k+K}$, we have
		\begin{align*}
			\tilde{V}(k) \ge &\left\{\sum_{t=K}^{k-1}\left[\left(\prod_{j=t+1}^{k-1}\left(1-\frac{\theta}{j}\right)\right)\frac{\theta}{t}\rho_w^{k-t}\right]\right\}^2\|\mx_*\|^2 \\
			\ge & \left\{\sum_{t=K}^{k-1}\left[\frac{(t+1)^{2\theta}}{k^{2\theta}}\frac{\theta}{t}\rho_w^{k-t}\right]\right\}^2\|\mx_*\|^2,
		\end{align*}
		where we invoked Lemma \ref{lem: product} for the second inequality. Then,
		\begin{align}
			\label{tilde V_lower_pre}
			\tilde{V}(k) & \ge \left[\frac{\theta\rho_w^k}{k^{2\theta}}\sum_{t=K}^{k-1}\frac{(t+1)^{2\theta-1}}{\rho_w^t}\right]^2\|\mx_*\|^2\notag\\
			& \ge \left[\frac{\theta\rho_w^k}{k^{2\theta}}\int_{t=K-1}^{k-1}\frac{(t+1)^{2\theta-1}}{\rho_w^t}dt\right]^2\|\mx_*\|^2.
		\end{align}
		Note that when $k\ge \frac{4\theta}{(-\ln \rho_w)}$,
		\begin{multline*}
			\int_{t=K-1}^{k-1}\frac{(t+1)^{2\theta-1}}{\rho_w^t}dt\ge \frac{2}{3}\left.\frac{(t+1)^{2\theta-1}}{(-\ln \rho_w)\rho_w^t}\right|_{t=K-1}^{k-1}\\
			=\frac{2k^{2\theta-1}}{3(-\ln \rho_w)\rho_w^{k-1}}-\frac{2K^{2\theta-1}}{3(-\ln \rho_w)\rho_w^{K-1}}
			\ge \frac{k^{2\theta-1}}{2(-\ln \rho_w)\rho_w^{k-1}}.
		\end{multline*}
		From (\ref{tilde V_lower_pre}),
		\begin{equation*}
			\tilde{V}(k)\ge  \left[\frac{\theta\rho_w}{2(-\ln \rho_w)k}\right]^2\|\mx_*\|^2=\Omega\left(\frac{n}{(1-\rho_w)^2}\right)\frac{1}{k^2},
		\end{equation*}
	{\sp where the equality is obtained from the Taylor expansion of $\ln \rho_w$ when $\rho_w\rightarrow 1$.}
		Since
		\begin{equation*}
			\frac{1}{n}\sum_{i=1}^n\bE[\|x_i(k)-x_*\|^2]=U(k)+\frac{V(k)}{n}\ge \Omega\left(\frac{1}{(1-\rho_w)^2}\right)\frac{1}{k^2},
		\end{equation*}
		setting this to be at most $\frac{\sp{\bar{M}}}{nk}$, we obtain that the transient time for DSGD to reach the asymptotic convergence rate is lower bounded by $\Omega\left(\frac{n}{(1-\rho_w)^2}\right)$, based on an argument similar to that of Corollary \ref{cor: transient_time}.
	\end{proof}
	
	\section{Numerical Examples}
	\label{sec: numerical}
	
	In this section, we provide two numerical example to verify and complement our theoretical findings. 
	
	\subsection{Ridge Regression}
	Consider the \emph{on-line} ridge regression problem, i.e.,
	\begin{equation}
		\label{Ridge Regression}
		\min_{x\in \mathbb{R}^{p}}f(x)=\frac{1}{n}\sum_{i=1}^nf_i(x)\left(=\mathbb{E}_{u_i,v_i}\left[\left(u_i^{\T} x-v_i\right)^2+\rho\|x\|^2\right]\right),
	\end{equation}
	where $\rho>0$ is a penalty parameter.
	Each agent $i$ collects data samples in the form of $(u_i,v_i)$ continuously where $u_i\in\mathbb{R}^p$ represent the features and $v_i\in\mathbb{R}$ are the observed outputs. Assume each $u_i\in[-0.5,0.5]^p$ is uniformly distributed, and $v_i$ is drawn according to $v_i=u_i^{\T} \tilde{x}_i+\varepsilon_i$, where $\tilde{x}_i$ are predefined parameters evenly located in $[0,10]^p$, and $\varepsilon_i$ are independent Gaussian random variables (noise) with mean $0$ and variance $0.01$.
	Given a pair $(u_i,v_i)$, agent $i$ can compute an estimated (unbiased) gradient of $f_i(x)$: $g_i(x,u_i,v_i)=2(u_i^{\T}x -v_i)u_i+2\rho x$.
	{\sp Problem (\ref{Ridge Regression}) has a unique solution $x_*$ given by 
		\begin{multline}
			\label{Ridge_opt}
			x_* = \left(\sum_{i=1}^n\mathbb{E}_{u_i}[u_iu_i^{\T}]+n\rho\mathbf{I}\right)^{-1}\sum_{i=1}^n\mathbb{E}_{u_i}[u_iu_i^{\T}]\tilde{x}_i\\
			=  \frac{1}{3}\left(\frac{1}{3}+\rho\right)^{-1}\frac{1}{n}\sum_{i=1}^n \tilde{x}_i.
		\end{multline}
	}
	
	Suppose $p=10$ and $\rho=1$. We compare the performance of DSGD (\ref{eq: x_i,k}) and the centralized implementation (\ref{eq: centralized}) for solving problem (\ref{Ridge Regression}) with the same stepsize policy $\alpha_k=20/(k+20),\forall k$, {\sp and the same initial solutions: $x_i(0)=\mathbf{0},\ \forall i$, (DSGD) and $x(0)=\mathbf{0}$ (SGD). It can be seen from \eqref{Ridge_opt} and the definition of $\tilde{x}_i$ that  $\sum_{i=1}^{n}\|x_i(0)-x_*\|^2 = \mathcal{O}(n)$. Moreover, $\nabla f_i(x_*)=2\mathbb{E}_{u_i,v_i}\left[(u_i^{\T}x_* -v_i)u_i\right]+2\rho x_* = 2\mathbb{E}_{u_i}[u_iu_i^{\T}](x_*-\tilde{x}_i)+2\rho x_* = \frac{2}{3}(x_*-\tilde{x}_i)+2\rho x_*$. Therefore, we have $\sum_{i=1}^{n}\|\nabla f_i(x_*)\|^2=\mathcal{O}(n)$.
	
In Fig. \ref{comparison_grid_circle}, we provide an illustration example that compares the performance of DSGD and SGD, assuming $n=25$. For DSGD, we consider two different network topologies: ring network topology as shown in Fig. \ref{Figures_ring}(a) and square grid network topology as shown in Fig. \ref{Figures_ring}(a). For both network topologies, we use Metropolis weights for constructing the maxing matrix $\mW$ (see \cite{nedic2018network}). It can be seen that DSGD performs asymptotically as well as SGD, while the time it takes for DSGD to catch up with SGD depends on the network topology. For grid networks which are better connected than rings, the corresponding transient time is shorter.
}
	\begin{figure}[htbp]
		\centering
		\includegraphics[width=2.9in]{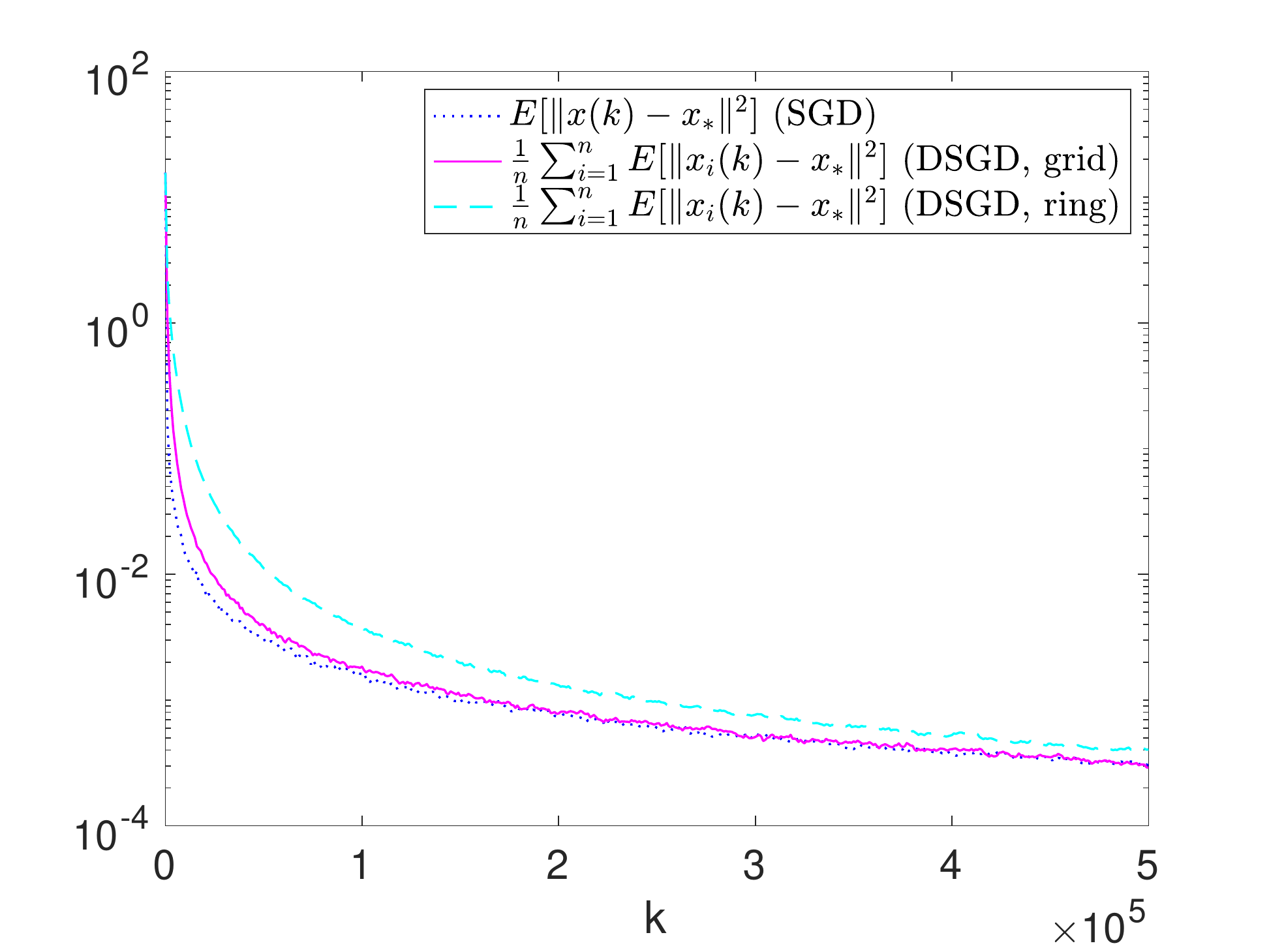}
		\caption{The performance comparison between DSGD and SGD for online Ridge regression ($n=25$). The results are averaged over $200$ Monte Carlo simulation.}
		\label{comparison_grid_circle}
	\end{figure}
	 
	{\sp To further verify the conclusions of Corollary \ref{cor: transient_time} and Theorem \ref{Thm: sharpness},} we define the transient time for DSGD as $\inf\{k: \frac{1}{n}\sum_{i=1}^n\bE[\|x_i(k)-x_*\|^2]\le 2\bE[\|x(k)-x_*\|^2]\}$. For DSGD, we first assume a ring network topology and plot the transient times for DSGD and $\frac{4n}{(1-\rho_w)^2}$ as a function of the network size $n$ in Fig. \ref{Figures_ring}(b). We then consider a square grid network topology as shown in Fig. \ref{Figures_ring}(a) and plot the transient times for DSGD and $\frac{7n}{(1-\rho_w)^2}$ in Fig. \ref{Figures_ring} (b).  It can be seen that the two curves in Fig. \ref{Figures_ring}(b) and Fig. \ref{Figures_grid}(b) are close to each other, respectively. This verifies the sharpness of Corollary \ref{cor: transient_time}.
	\begin{figure}[htbp]
		\centering
		\subfloat[Ring network topology.]{\includegraphics[width=1.9in]{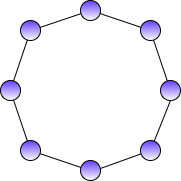}} \label{Figure_ring_graph}
		\subfloat[Transient times for the ring network topology.]{\includegraphics[width=2.8in]{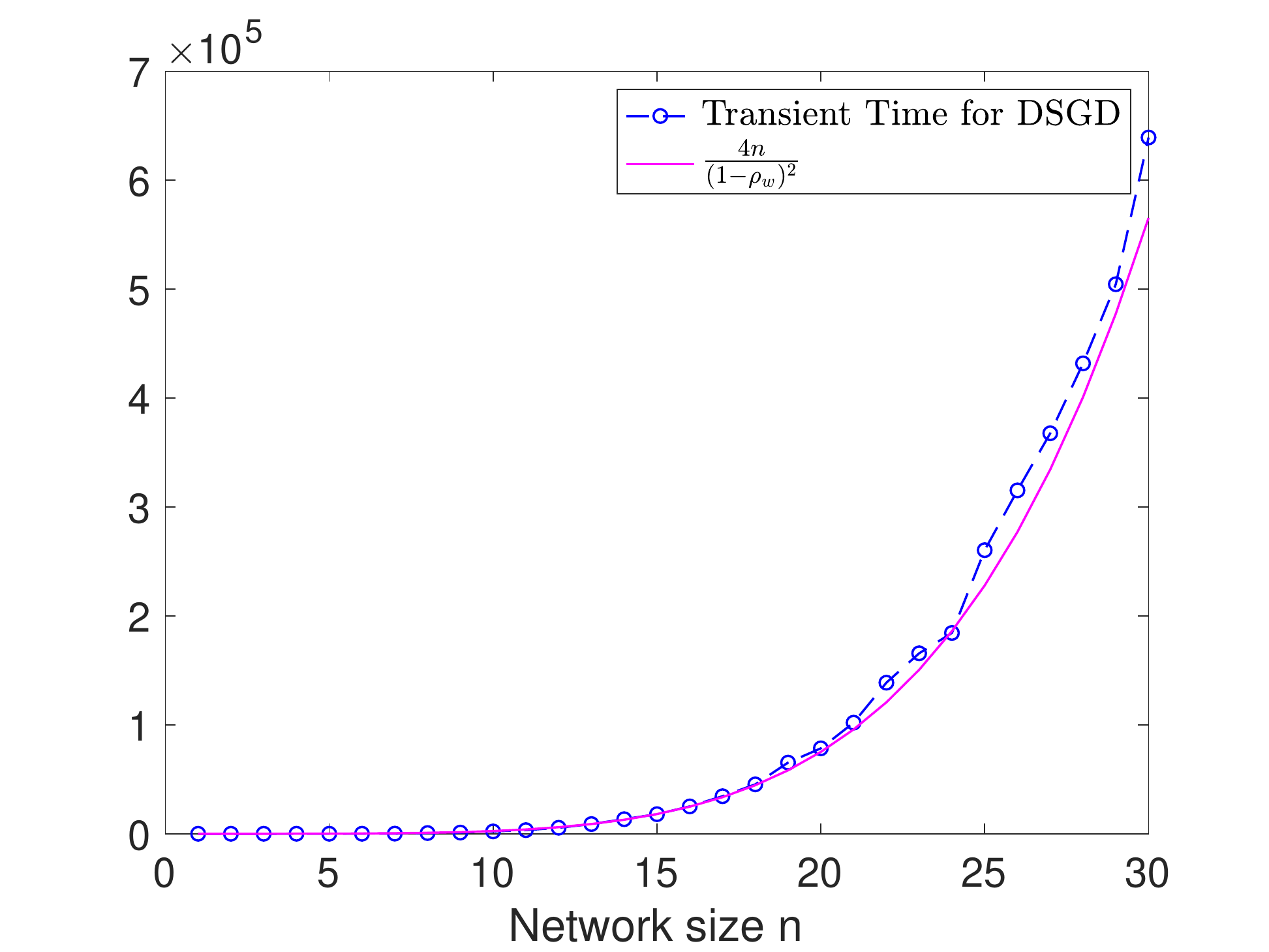}} \label{Figure_transient}
		\caption{Comparison of the transient times for DSGD and $\frac{4n}{(1-\rho_w)^2}$ as a function of the network size $n$ for the ring network topology. The expected errors are approximated by averaging over $200$ simulation results.}
		\label{Figures_ring}
	\end{figure}
	\begin{figure}[htbp]
		\centering
		\subfloat[Square grid network topology.]{\includegraphics[width=1.9in]{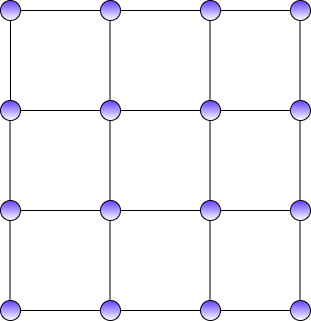}} \label{Figure_grid}
		\subfloat[Transient times for the square grid network topology.]{\includegraphics[width=2.8in]{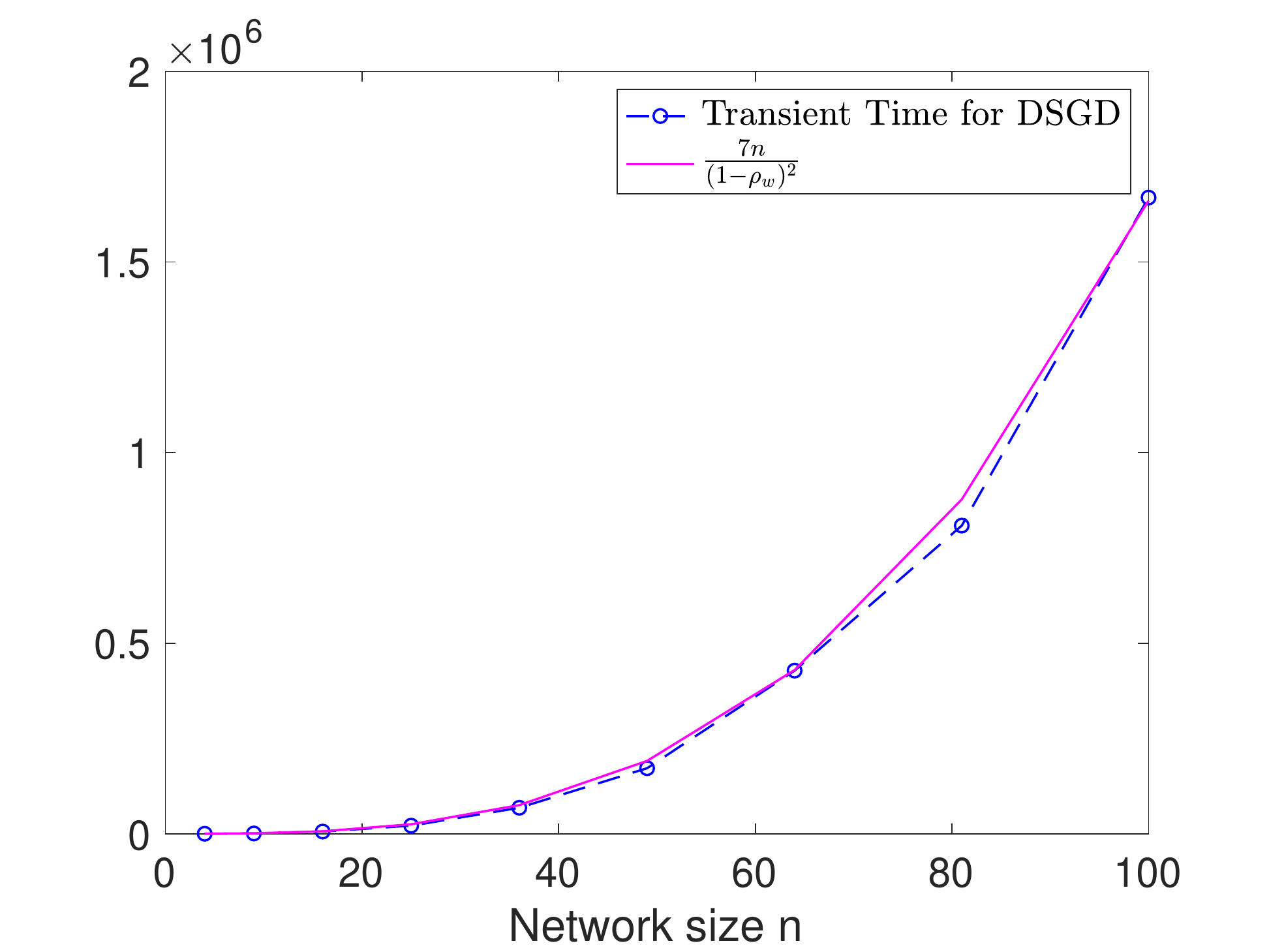}} \label{Figure_transient_grid}
		\caption{Comparison of the transient times for DSGD and $\frac{7n}{(1-\rho_w)^2}$ as a function of the network size $n$ for the square grid network topology ($n=4,9,16,25,36,49,64,81,100$). The expected errors are approximated by averaging over $200$ simulation results.}
		\label{Figures_grid}
	\end{figure}
	
	\subsection{Logistic Regression}
	{\sp Consider the problem of classification on the MNIST dataset of handwritten digits (http://yann.lecun.com/exdb/mnist/). In particular, we classify digits $1$ and $2$ using logistic regression.\footnote{\sp The problem can be extended to classifying all $10$ handwritten digits with multinomial logistic regression.} There are $12700$ data points in total where each data point is a pair $(u,v)$ with $u\in\mathbb{R}^{785}$ being the image input and $v\in\{0,1\}$ being the label.\footnote{\sp Digit $1$ is represented by label $0$ and digit $2$ is represented by label $1$.}

Suppose each agent $i\in\mathcal{N}$ possesses a distinct local dataset $\mathcal{S}_i$ that is randomly taken from the database.  To apply logistic regression for classification, we solve the following optimization problem based on all the agents' local datasets:
\begin{equation}
	\label{Logistic Regression}
	\min_{x\in \mathbb{R}^{785}}f(x)=\frac{1}{n}\sum_{i=1}^n f_i(x),
\end{equation}
where
\begin{multline*}
	f_i(x)  := \frac{1}{|\mathcal{S}_i|}\sum_{j\in \mathcal{S}_i}\left[\log(1+\exp(-x^{\T}u_j)) + (1-v_j)x^{\T}u_j\right]\\
	+\frac{\lambda}{2}\|x\|^2,
\end{multline*}
where $\lambda$ is the regularization parameter.\footnote{\sp The obtained optimal solution $x_*$ of problem \eqref{Logistic Regression} can then be used for predicting the label for any image input $u$ through the decision function $h(u):=\frac{1}{1+\exp(-x_*^{\T}u)}$.}
Given any solution $x$,  agent $i$  is able to compute an unbiased estimate of $\nabla f_i(x)$ using one (or a minibatch of) randomly chosen data point $(u_i,v_i)$ from $\mathcal{S}_i$, that is,
\begin{equation*}
	g_i(x,u_i,v_i) = \frac{-u_j}{1+\exp(x^{\T}u_j)} + (1-v_j)u_j+\lambda x.
\end{equation*}

In the experiments, suppose each local dataset $\mathcal{S}_i$ contains $50$ data points, and $\lambda = 1$. At each iteration of the DSGD algorithm, agent $i$ computes a stochastic gradient of $f_i(x_i(k))$ with one randomly chosen data point from $\mathcal{S}_i$. We compare the performance of DSGD (\ref{eq: x_i,k}) and centralized SGD (\ref{eq: centralized}) for solving problem \eqref{Logistic Regression} with the same stepsize policy $\alpha_k = 6/(k+20),\forall k$,  and the same initial solutions: $x_i(0)=\mathbf{0},\ \forall i$, (DSGD) and $x(0)=\mathbf{0}$ (SGD). It can be numerically verified that $\sum_{i=1}^{n}\|x_i(0)-x_*\|^2=\mathcal{O}(n)$ and $\sum_{i=1}^{n}\|\nabla f_i(x_*)\|^2=\mathcal{O}(n)$. 

The transient time for DSGD is defined in the same way as in the ridge regression example. In Fig. 4 and Fig. 5, we plot the transient times for DSGD as a function of the network size $n$ for ring and grid networks, respectively. We find that the curves are close to $\frac{n}{4(1-\rho_w)^{1.5}}$, rather than a multiple of $\frac{n}{(1-\rho_w)^{2}}$, implying that the experimental results are better than the theoretically derived worst-case performance  given in Corollary \ref{cor: transient_time}. Hence in practice, the performance of the DSGD algorithm depends on the specific problem instances and can be better than the worst-case situation  in terms of transient times. 
}
\begin{figure}[htbp]
	\centering
	\includegraphics[width=2.8in]{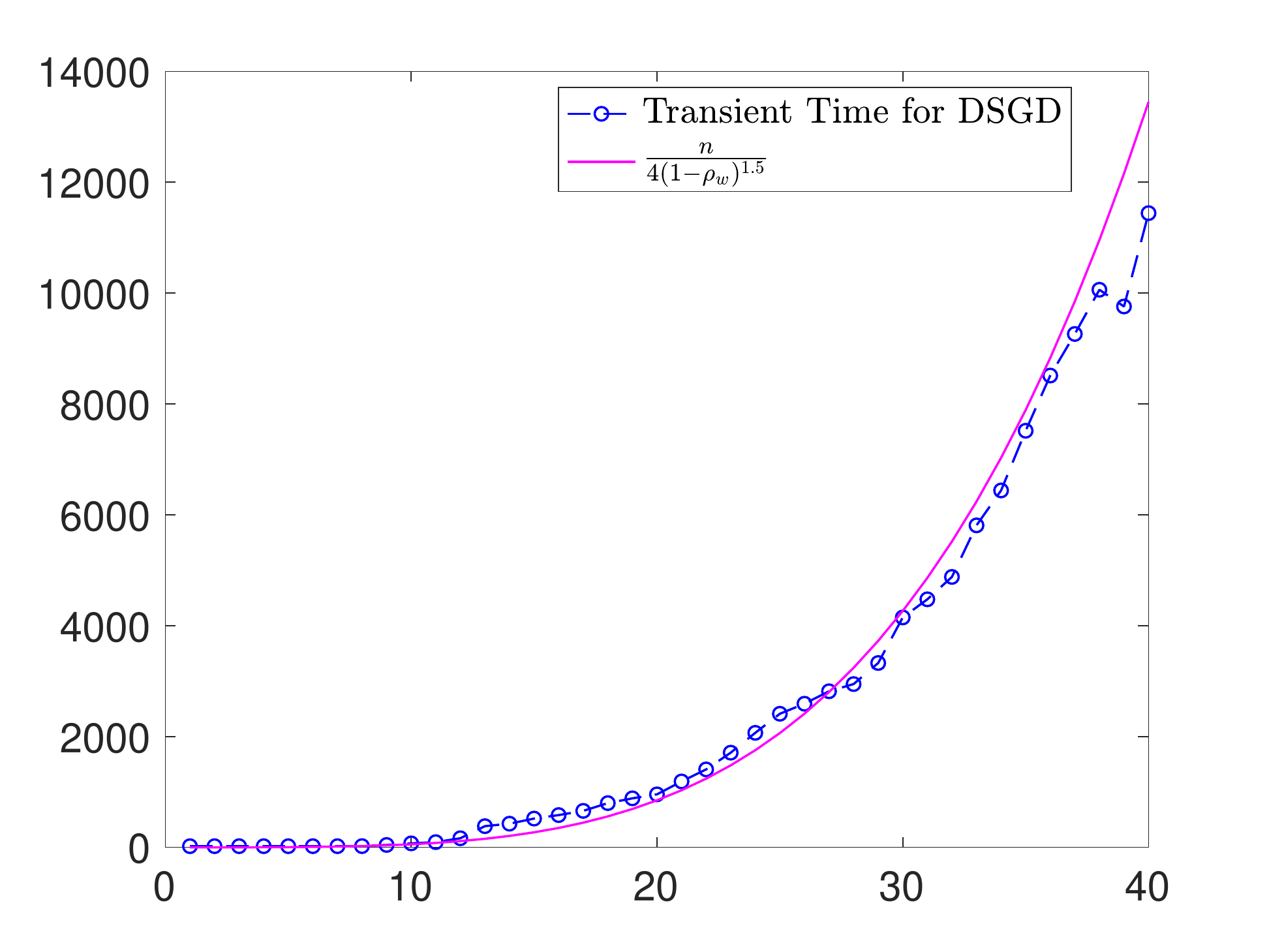} \label{Figure_logistic_transient}
	\caption{Comparison of the transient times for DSGD and $\frac{n}{4(1-\rho_w)^{1.5}}$ as a function of the network size $n$ for the ring network topology. The expected errors are approximated by averaging over $200$ simulation results.}
	\label{Figures_logistic_ring}
\end{figure}
\begin{figure}[htbp]
	\centering
	\includegraphics[width=2.8in]{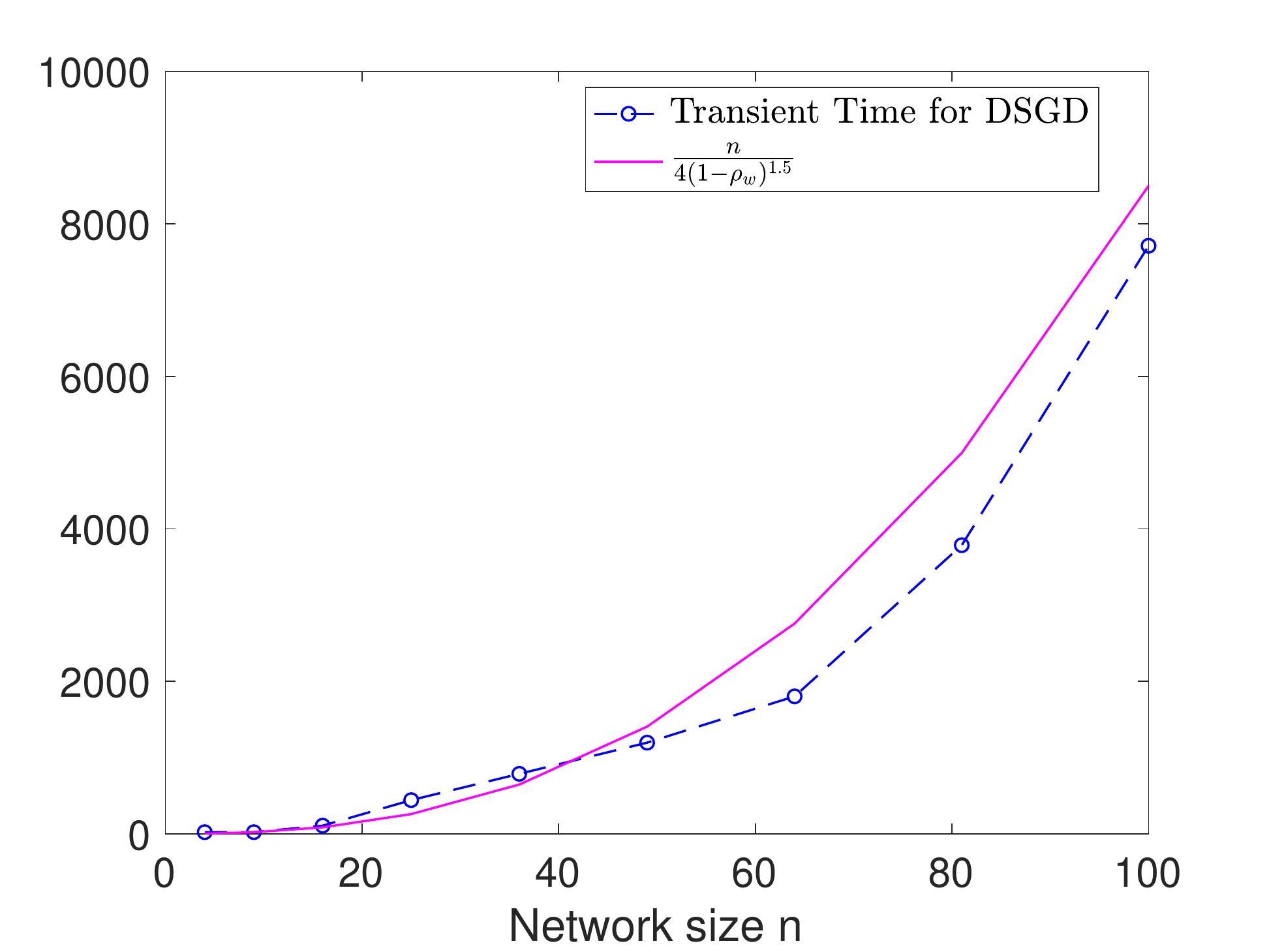} \label{Figure_logistic_transient_grid}
	\caption{Comparison of the transient times for DSGD and $\frac{n}{4(1-\rho_w)^{1.5}}$ as a function of the network size $n$ for the grid network topology ($n=4,9,16,25,36,49,64,81,100$). The expected errors are approximated by averaging over $200$ simulation results.}
	\label{Figures_logistic_grid}
\end{figure}

	\section{Conclusions}
	\label{sec: conclusions}
	
	This paper is devoted to the non-asymptotic analysis of network independence for the distributed stochastic gradient descent (DSGD) method. We show that the algorithm asymptotically achieves the optimal network independent convergence rate compared to SGD, and identify the non-asymptotic convergence rate as a function of characteristics of the objective functions and the network. In addition, we compute the time needed for DSGD to reach its asymptotic rate of convergence and prove the sharpness of the obtained result. Future work will consider more general problems such as nonconvex objectives and constrained optimization.

	\ifCLASSOPTIONcaptionsoff
	\newpage
	\fi

	% trigger a \newpage just before the given reference
	% number - used to balance the columns on the last page
	% adjust value as needed - may need to be readjusted if
	% the document is modified later
	%\IEEEtriggeratref{8}
	% The "triggered" command can be changed if desired:
	%\IEEEtriggercmd{\enlargethispage{-5in}}
	
	% references section
	
	% can use a bibliography generated by BibTeX as a .bbl file
	% BibTeX documentation can be easily obtained at:
	% http://mirror.ctan.org/biblio/bibtex/contrib/doc/
	% The IEEEtran BibTeX style support page is at:
	% http://www.michaelshell.org/tex/ieeetran/bibtex/
	\bibliographystyle{IEEEtran}
	% argument is your BibTeX string definitions and bibliography database(s)
	\bibliography{mybib}
	%
	% <OR> manually copy in the resultant .bbl file
	% set second argument of \begin to the number of references
	% (used to reserve space for the reference number labels box)
	%\begin{thebibliography}{1}
	%
	%\bibitem{IEEEhowto:kopka}
	%H.~Kopka and P.~W. Daly, \emph{A Guide to \LaTeX}, 3rd~ed.\hskip 1em plus
	%  0.5em minus 0.4em\relax Harlow, England: Addison-Wesley, 1999.
	%
	%\end{thebibliography}
	
	% biography section
	% 
	% If you have an EPS/PDF photo (graphicx package needed) extra braces are
	% needed around the contents of the optional argument to biography to prevent
	% the LaTeX parser from getting confused when it sees the complicated
	% \includegraphics command within an optional argument. (You could create
	% your own custom macro containing the \includegraphics command to make things
	% simpler here.)
	%\begin{IEEEbiography}[{\includegraphics[width=1in,height=1.25in,clip,keepaspectratio]{mshell}}]{Michael Shell}
	% or if you just want to reserve a space for a photo:
	
	\begin{IEEEbiographynophoto}{Shi Pu}
		is currently an assistant professor in the School of Data and Decision Science, The Chinese University of Hong Kong, Shenzhen, China. He is also affiliated with Shenzhen Research Institute of Big Data. He received a B.S. Degree from Peking University, in 2012, and a Ph.D. Degree in Systems Engineering from the University of Virginia, in 2016. He was a postdoctoral associate at the University of Florida, from 2016 to 2017, a postdoctoral scholar at Arizona State University, from 2017 to 2018, and a postdoctoral associate at Boston University, from 2018 to 2019. His research
		interests include distributed optimization, network science, machine learning, and game theory.
	\end{IEEEbiographynophoto}
	
	% if you will not have a photo at all:
	\begin{IEEEbiographynophoto}{Alex Olshevsky}
		received the B.S. degrees in applied mathematics and electrical engineering from Georgia Tech and the Ph.D. degree in EECS from MIT. He is currently an Associate Professor at the ECE department at Boston University. Dr. Olshevsky is a recipient of the NSF CAREER Award, the AFOSR Young Investigator Award, the INFORMS Prize for the best paper on the interface of operations research and computer science, a SIAM Award for annual paper from the SIAM Journal on Control and Optimization chosen to be reprinted in SIAM Review, and an IMIA award for best paper on clinical informatics.
	\end{IEEEbiographynophoto}
	
	%\begin{IEEEbiography}[{\includegraphics[width=1in,height=1.25in,clip,keepaspectratio]{XJM}}]{Jinming Xu}
	\begin{IEEEbiographynophoto}{Ioannis~Ch.~Paschalidis}
	(M'96--SM'06--F'14) received a Ph.D. in EECS from the
Massachusetts Institute of Technology, Cambridge, MA, USA, in 1996.  He is a
Professor and Data Science Fellow at Boston University, Boston, MA and the Director of the Center
for Information and Systems Engineering. His research interests lie in the fields of
systems and control, networks, applied probability, optimization, operations
research, computational biology, and medical informatics. He is a recipient of the
NSF CAREER award and several best paper awards. During 2013--2019 he was the founding
Editor-in-Chief of the IEEE Transactions on Control of Network Systems.
	\end{IEEEbiographynophoto}
	
	% insert where needed to balance the two columns on the last page with
	% biographies
	
	% You can push biographies down or up by placing
	% a \vfill before or after them. The appropriate
	% use of \vfill depends on what kind of text is
	% on the last page and whether or not the columns
	% are being equalized.
	
	%	\vfill
	
	% Can be used to pull up biographies so that the bottom of the last one
	% is flush with the other column.
	%	\enlargethispage{-5in}

	% that's all folks
\end{document}